\documentclass{article}

\usepackage{geometry}
\usepackage{amsmath,amssymb,comment,enumitem, array}
\usepackage{color}
\usepackage{parskip}
\usepackage{parskip}
\usepackage{setspace}
\usepackage{graphicx}
\usepackage{mathtools}
\usepackage{subfig}
\usepackage{comment}
\usepackage{mathabx}
\usepackage{mathrsfs}
\usepackage{setspace}
\usepackage{bm}
\usepackage[normalem]{ulem}
\usepackage{cancel}
\usepackage{csquotes}
\usepackage[backend=biber, maxbibnames=99]{biblatex}
\usepackage[colorlinks=true, allcolors=blue]{hyperref}
\hypersetup{colorlinks, citecolor=red, filecolor=black, linkcolor=blue, urlcolor=blue}
\usepackage{pgfplots}
\pgfplotsset{compat=1.18}
\usepackage{graphics}
\usepackage{graphicx}
\usepackage{tikz}
\usepackage{tcolorbox}
\usepackage{enumitem}
\usepackage{subfiles}
\usepackage{ytableau}
\usepackage{thmtools}
\usepackage{thm-restate}
\usepackage{amsthm}
\usepackage{multicol}
\usetikzlibrary{patterns}

\newcommand\sm{\setminus}

\newtheorem{thm}{Theorem}
\newtheorem{lemma}[thm]{Lemma}
\newtheorem{prop}[thm]{Proposition}
\newtheorem{claim}[thm]{Claim}

\newtheorem{prob}[thm]{Problem}
\newtheorem{cor}[thm]{Corollary}

\theoremstyle{definition}
\newtheorem{defn}[thm]{Definition}
\newtheorem{eg}[thm]{Example}

\def\john#1{\noindent
\textcolor{blue}
{\textsc{(John:}
\textsf{#1})}}

\title{The VC-dimension of strongly regular graphs}
\author{Isabel Byrne\footnote{Department of Mathematical Sciences, University of Delaware, Newark, DE 19716, USA} \and John Byrne\footnote{Department of Mathematics, University of California, San Diego, La Jolla, CA 92093, USA} \and Sebastian M. Cioab\u a\footnote{Department of Mathematical Sciences, University of Delaware, Newark, DE 19716, USA}}
\date{\today}

\begin{document}

\maketitle

\begin{abstract}
    A graph $G$ is $n$-existentially closed or $n$-\textit{e.c.} if, for all subsets $S\subseteq V(G)$ with $|S|=n$ and for all partitions $S=A\sqcup B$, there exists a vertex in $V(G)\sm S$ adjacent to all vertices in $A$ and no vertices in $B$. We study the minimum number of edges $m(v,n)$ of a $v$-vertex $n$-e.c. graph, and show that $m(v,2)=3v+O(1)$ while $m(v,n)=\Theta(v\log v)$ for fixed $n\ge 3$. The latter result uses a connection to binary covering arrays.

    A related parameter is the \textit{VC-dimension} of $G$, defined as the size of the largest subset of vertices shattered by the neighborhoods of vertices in $G$. We initiate systematic study of the VC-dimensions of strongly regular graphs (SRGs). We characterize the sufficiently large SRGs with VC-dimension 2. Furthermore, we determine the VC-dimension of sufficiently large Latin square graphs and of all SRGs of order at most 28, and prove upper bounds for SRGs arising from orthogonal arrays or Steiner systems.
\end{abstract}

\section{Introduction}

In this paper we study two related graph parameters. A graph $G$ is said to be \textit{$n$-existentially closed} ($n$-e.c.) if, for all subsets $S\subseteq V(G)$ with $|S|=n$ and for all partitions $S=A\sqcup B$, there exists a vertex in $V(G)\sm S$ adjacent to all vertices in $A$ and no vertices in $B$. The $n$-e.c. property was introduced by Erd\H{o}s and R\'enyi in 1963 \cite{erdos1963asymmetric}, when they proved that the random graph $G_{n,p}$ is $n$-e.c. with high probability. The \textit{VC-dimension} of $G$ is the maximum size of a subset $S\subseteq V(G)$ such that, for all partitions $S=A\sqcup B$, there exists a vertex in $V(G)$ adjacent to all vertices in $A$ and no vertices in $B$. Evidently, if $G$ is $n$-existentially closed, then the VC-dimension of $G$ is at least $n$. The converse is not true, and indeed there exist graphs with arbitrary VC-dimension that are not even 1-e.c.. (Note that $G$ is 1-e.c. if and only if the minimum degree of both $G$ and its complement are at least 1.)

The $n$-e.c. property gives rise to a natural extremal problem: determine $m_{ec}(n)$, the minimum order of an $n$-e.c. graph. The only exact values known are $m_{ec}(1)=4$ and $m_{ec}(2)=9$ \cite{caccetta1985property}, while the best bounds for $n=3$ are $24\le m_{ec}(3)\le 28$ \cite{bonato2001adjacency,gordinowicz2010search}. For the asymptotics as $n\to\infty$, the best known upper bound is $m_{ec}(n)=O(n^22^n)$, from the random graph $G_{n,1/2}$ \cite{erdos1963asymmetric}. Caccetta, Erd\H{o}s, and Vijayan \cite{caccetta1985property} showed that $m_{ec}(n)=\Omega(n2^n)$ and conjectured that $\lim_{n\to\infty}m_{ec}(n)/(n2^n)$ exists.

The random graph $G_{n,1/2}$, as well as most explicit constructions of $n$-e.c. graphs (see the survey \cite{bonato2009search}), typically has an edge-density near $1/2$ and so is \textit{dense}. In this paper, we study \textit{sparse} $n$-e.c. graphs. Accordingly, we define $m(v,n)$ to be the minimum number of edges in an $n$-e.c. graph of order $v$, and study the growth of this function for $n$ fixed and $v\to\infty$. (Note this is equivalent to studying the \textit{maximum} number of edges in such a graph.) In this paper, we show that $m(v,2)=3v+O(1)$, while for $n\ge 3$, $m(v,n)=\Theta(v\log v)$. Additionally, the $n\ge 3$ case is closely related to the problem of constructing $n$-\textit{strong binary covering arrays.}

When considering VC-dimension instead of $n$-e.c., these types of extremal problems are no longer interesting. For example, it is obvious that the minimum order of a graph with VC-dimension $n$ is $2^n$. Instead, our main goal here is the exact determination of the VC-dimension of various given graphs. 

In 1971, Vapnik and Chervonenkis \cite{vapnik1971uniform} introduced the notion of VC-dimension in the context of statistics. The VC-dimension of a set system is a measure of the complexity of a set system, and it has found applications in such areas as machine learning, logic, real algebraic geometry, and computational geometry (see e.g. \cite[Section 14.4]{alon2016probabilistic}).

There is a growing literature on the VC-dimension of graphs, motivated by several distinct applications. The series of papers \cite{fitzpatrick2024vc,iosevich2025vc,iosevich2023dot} concerns learning problems related to point configurations in both finite and real vector spaces. The set systems they consider consist of certain neighborhood sets in certain graphs, and this graph-theoretic perspective was pursued further in \cite{housholder2026vc,pham2025vc}. Combinatorial applications of graph VC-dimension were obtained in \cite{alon2007efficient,fox2020bounded,fox2019erdHos,janzer2024zarankiewicz,lovasz2010regularity,nguyen2025induced}. Of particular interest are graph classes with bounded VC-dimension; for such classes, it is often possible to prove bounds on some parameter which greatly improve what is known for general graphs. Algorithmic applications for these classes have also been obtained: see \cite{coudert2024practical} for an overview.

In this paper we study the VC-dimension of \textit{strongly regular graphs}, a class which includes cases of the aforementioned Hamming and Johnson graphs, as well as many other graphs of fundamental importance in combinatorics and algebra. Conceptually, the aim of this work is to observe how a \textit{regularity}/\textit{balance} condition affects a measure of the \textit{complexity} of a graph.

Several authors have obtained results about the VC-dimension of specific families of SRGs. Benediktsson \cite{benediktsson2021model} obtained results on the VC-dimension of Hamming graphs and Johnson graphs. In particular, the VC-dimension of a sufficiently large lattice graph is 3, and the VC-dimension of a sufficiently large triangular graph is 4. In \cite{mcdonald2025vc}, it was proved that the VC-dimension of the Paley graph of order $q$ is at least $(1/2+o(1))\log_2(q)$ and conjectured that the true value is $(1+o(1))\log_2(q)$.

Other bounds on the VC-dimension of particular SRGs follow from constructions of $n$-e.c. SRGs. Bonato \cite{bonato2009search} noted that ``most of the known explicit $n$-e.c. graphs are strongly regular." From \cite{kisielewicz2004pseudo} we have that the VC-dimension of the Peisert graph of order $q$ is at least $(1/8+o(1))\log_2(q)$. Cameron and Stark \cite{cameron2002prolific} gave a randomized construction of $v$-vertex SRGs based on affine designs, which have VC-dimension at least $(1/2+o(1))\log_2(v)$.

It is likely difficult or impossible to characterize all $n$-e.c. graphs or all graphs with VC-dimension $n$, even for very small $n$; however, in this paper we accomplish this for VC-dimension $2$ when restricting to large enough SRGs. Precisely, every primitive SRG has VC-dimension at least 2, and we characterize (in terms of their parameters $(v,k,\lambda,\mu)$), for sufficiently large order $v$, the SRGs which have VC-dimension exactly 2. Additionally, we determine the VC-dimension of sufficiently large Latin square graphs and of all SRGs of order at most 28. We also obtain bounds on the VC-dimension of other families of SRGs, namely orthogonal array graphs and Steiner graphs, in terms of the number of rows/block size $k$. These two families are of particular interest, in view of the results of Neumaier and Sims (see \cite[Thm.8.6.3,Thm.8.6.4]{brouwer2022strongly} and \cite{koolen26}): for a fixed integer $m\ge 2$, all but finitely many primitive SRGs with smallest eigenvalue $-m$ are orthogonal array graphs or Steiner graphs, and the parameter $k$ representing the number of rows/block size is equal to $m$. McKay and Pike proved that if a Steiner graph is $n$-e.c. then $n$ is bounded by the block size $k$ (\cite{mckay2007existentially}, Theorem 1). Our result on Steiner graphs strengthens this result by replacing $n$ with the VC-dimension, at the cost of a allowing a larger function of $k$. As pointed out by Istv\'an Tomon, it already follows from Lemma 8.2 of \cite{balla2026factorization} and Proposition 2.2 of \cite{goh2025block} that the VC-dimension of any graph with smallest eigenvalue $-m$ is at most $(2m+1)^2$; and furthermore, using Theorem 10.1 of \cite{jin2025small}, one can prove the much stronger bound of $C\log_2(m+2)$ for some absolute constant $C$ \cite{Tomon2026}. Thus, the import of our result is for small values of $k$, though the value of ``small" here is unknown.

\section{Notation}

Let $G$ be a graph. For $U\subseteq V(G)$, we write $G[U]$ for the subgraph induced by $U$; and for $x\in V(G)$, we use $N_U(x)$ for $N(x)\cap U$ and $d_U(x)$ for $|N_U(x)|$. We write $N_i(x)$ for $\{y\in V(G):d(x,y)=i\}.$ We also use $e(G)$ for the number of edges of $G$ and $e(U)$ for $e(G[U])$; for $W\subseteq V$, we write $e(U,W)$ for $|\{(u,w):u\in U,w\in W,u\sim w\}|.$ We denote by $P_k$ the path on $k$ \textit{vertices} and by $K_{n_1,\ldots,n_p}$ the complete multipartite graph with part size $n_1,\ldots,n_p$.

Some of our notation is nonstandard. Given a graph $G$ and $x_1,\ldots, x_i,y_1,\ldots,y_j\in V(G)$, write $N(x_1,\ldots, x_i,\overline{y_1},\ldots,\overline{y_j})$ for the set of vertices in $V\sm \{x_1,\ldots,x_i,y_1,\ldots,y_j\}$ which are adjacent to all of $x_1,\ldots, x_i$ and none of $y_1,\ldots, y_j$; and $N'(x_1,\ldots,x_i,\overline{y_1},\ldots,\overline{y_j})$ for the set of vertices in $V$ which are adjacent to all of $x_1,\ldots,x_i$ and none of $y_1,\ldots,y_j$. We write $d(x_1,\ldots,x_i,\overline{y_1},\ldots,\overline{y_j})$ and $d'(x_1,\ldots,x_i,\overline{y_1},\ldots,\overline{y_j})$ for the size of each respective set. This notation may be abbreviated as $N(X,\overline Y)$, etc. when $X=\{x_1,\ldots,x_i\}$ and $Y=\{y_1,\ldots,y_j\}$. Using this notation, the graph properties we study can be defined succinctly as follows:
\begin{enumerate}
    \item $G$ is $n$-e.c. if and only if $N(x_1,x_2,\ldots,x_i,\overline{x_{i+1}},\overline{x_{i+2}},\ldots,\overline{x_n})\ne\emptyset$ for all distinct $x_1,\ldots,x_n\in V$.
    \item The VC-dimension of $G$ is the maximum size of a set $S\subseteq V$ such that $N'(A,\overline{S\sm A})\ne\emptyset$ for all $A\subseteq S$.
\end{enumerate}

Whenever we make a statement such as ``choose $x_1,x_2\in V$," we implicitly assume that $x_1,x_2$ are distinct. When studying $n$-e.c. and VC-dimension, we believe it is sufficiently clear from context when vertices are assumed to be distinct, that it is worth omitting this condition to save space.

\section{n-e.c. graphs}

Define $m(v,n)$ to be the minimum size of a $v$-vertex $n$-e.c. graph. Our goal is to study the asymptotics of $m(v,n)$ for fixed $n$ and $v\to\infty$. Note that, since $G$ is $n$-e.c. if and only if its complement $\overline{G}$ is $n$-e.c., this is equivalent to studying the \textit{maximum} number of edges in an $n$-e.c. graph. The following lemma gives a lower bound on $m(v,n)$ which is linear in $v$.

\begin{lemma} \label{Lemma minimum degree}
    Suppose that $G$ is $n$-e.c. for some $n\ge 2$. Then $\delta(G)\ge m_{ec}(n-1).$
\end{lemma}
\begin{proof}
    Let $x\in V(G)$. Using \cite{bonato2009search}, $G[N(x)]$ is $(n-1)$-e.c.. Thus, $d(x)=|N(x)|\ge m_{ec}(n-1)$.
\end{proof}

\subsection{The case $n=1$}

\begin{thm}
    We have $m(v,1)=\lceil v/2\rceil$ for every $v\ge 4$.
\end{thm}
\begin{proof}
    We first prove the upper bound. If $v\ge 4$ is even, then a matching on $v$ vertices is 1-e.c., so $m(v,1)\ge v/2=\lceil v/2\rceil.$ If $v$ is odd, then note that the union of $P_3$ with a matching on the remaining vertices is 1-e.c., hence, $m(v,1)\ge 2+(v-3)/2=\lceil v/2\rceil.$
    
    We now prove the lower bound. If $G$ is 1-e.c., then every vertex is contained in an edge. By the union bound, this requires at least $v/2$ edges, so $e(G)\ge\lceil v/2\rceil$.
\end{proof}

\subsection{The case $n=2$}

In this section we prove the following.

\begin{thm} \label{Theorem m(v,2)}
    We have $3v-5625\le m(v,2)\le 3v+25$ for all $v\ge 3201$.
\end{thm}

We did not optimize the constants 5625 and 3201. We will first prove the upper bound, then the lower bound.

\subsubsection{Upper bound}

The following construction covers orders of the form $v=6k+3$.

\begin{defn} \label{Construction Dl}
    Let $\ell\ge 12$ be an integer such that $6|\ell$. We define a graph $D_\ell$ on $v=\ell+3$ vertices are follows. The vertex set is $V=\{0,\ldots,\ell-1\}\cup\{u_1,u_2,u_3\}$. Let $L_1=\{0,\ldots,\ell/3-1\}$, $L_2=\{\ell/3,\ldots,2\ell/3-1\}$, $L_3=\{2\ell/3,\ldots,\ell-1\}$ and $L=\{1,\ldots,\ell-1\}$. Set $U=\{u_1,u_2,u_3\}$. The adjacencies are defined as follows:
    \begin{itemize}
        \item for all $m\in L$, if $m$ is even then $m\sim m+1$;
        \item for all $m\in L_1\cup L_2$, if $m$ is odd then $m\sim m+\ell/3-1$;
        \item for all $m\in L_3$, if $m$ if odd then $m\sim m-2\ell/3-1$.
        \item for each $i=1,2,3$, for all $m\in L_i$, $m\sim u_j$ for $j\ne i$.
    \end{itemize}
\end{defn}

\begin{figure}
\begin{center}
\begin{tikzpicture}[scale=1]
    \draw (2.5,0)--(1,2)--(3.5,0);
    \draw (4.5,0)--(1,2)--(5.5,0);
    \draw (2.5,0)--(7,2)--(3.5,0);
    \draw (4.5,0)--(7,2)--(5.5,0);
    \draw (7.5,0)--(1,2)--(8.5,0);
    \draw (9.5,0)--(1,2)--(10.5,0);
    \draw (7.5,0)--(4,2)--(8.5,0);
    \draw (9.5,0)--(4,2)--(10.5,0);
    \draw (-2.5,0)--(4,2)--(-1.5,0);
    \draw (-0.5,0)--(4,2)--(0.5,0);
    \draw (-2.5,0)--(7,2)--(-1.5,0);
    \draw (-0.5,0)--(7,2)--(0.5,0);
    \draw (-2.5,0)--(-1.5,0);
    \draw (-0.5,0)--(0.5,0);
    \draw (2.5,0)--(3.5,0);
    \draw (4.5,0)--(5.5,0);
    \draw (7.5,0)--(8.5,0);
    \draw (9.5,0)--(10.5,0);
    \draw (-1.5,0) .. controls (0.5,-1) .. (2.5,0);
    \draw (0.5,0) .. controls (2.5,-1) .. (4.5,0);
    \draw (3.5,0) .. controls (5.5,-1) .. (7.5,0);
    \draw (5.5,0) .. controls (7.5,-1) .. (9.5,0);
    \draw (-2.5,0) .. controls (3,-2.3) .. (8.5,0);
    \draw (-0.5,0) .. controls (5,-2.3) .. (10.5,0);
    
    \draw[fill=black] (1,2) circle (2pt) node[above] {$u_1$};
    \draw[fill=black] (4,2) circle (2pt) node[above] {$u_2$};
    \draw[fill=black] (7,2) circle (2pt) node[above] {$u_3$};
    \draw[fill=black] (3.5,0) circle (2pt) node[below] {5};
    \draw[fill=black] (2.5,0) circle (2pt) node[below] {4};
    \draw[fill=black] (4.5,0) circle (2pt) node[below] {6};
    \draw[fill=black] (5.5,0) circle (2pt) node[below] {7};
    \draw[fill=black] (0.5,0) circle (2pt) node[below] {3};
    \draw[fill=black] (-0.5,0) circle (2pt) node[below] {2};
    \draw[fill=black] (-1.5,0) circle (2pt) node[below] {1};
    \draw[fill=black] (-2.5,0) circle (2pt) node[below] {0};
    \draw[fill=black] (7.5,0) circle (2pt) node[below] {8};
    \draw[fill=black] (8.5,0) circle (2pt) node[below] {9};
    \draw[fill=black] (9.5,0) circle (2pt) node[below] {10};
    \draw[fill=black] (10.5,0) circle (2pt) node[below] {11};
    
\end{tikzpicture}
\caption{The graph $D_\ell$ with $\ell=12$.} \label{Figure Dl}
\end{center}
\end{figure}
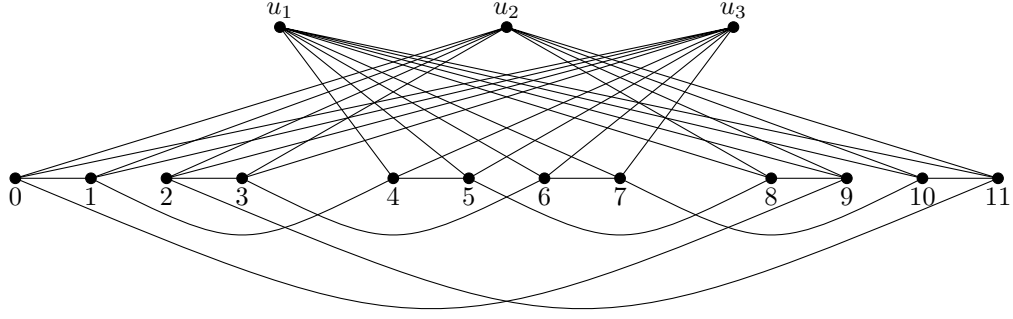

We observe the following properties of $D_\ell$:
\begin{itemize}
    \item[(1)] The set $L$ induces $\ell/6$ disjoint cycles of length $6$.
    \item[(2)] For each vertex $m\in L_i$, $m$ is adjacent to one vertex in $L_i$ and one vertex in $L_j$ for some $j\ne i$. 
    \item[(3)] $e(D_\ell)=3\ell$.
\end{itemize} 

\begin{prop}
    The graph $D_\ell$ is 2-e.c.
\end{prop}
\begin{proof}
    Let $x,y$ be distinct vertices of $D_\ell$.\\
    \textit{Case 1:} $x,y\in L$. Assume that $x\in L_i$ and $y\in L_j$. There exists $k\in\{1,2,3\}\sm\{i,j\}$. Then $u_k\in N(x, y)$. Since $L$ induces $\ell/6$ disjoint cycles of length 6, it is easy to find vertices in $N(x,\overline y)$, $N(\overline x,y)$ and $N(\overline x,\overline y)$ which are in $L$.\\
    \textit{Case 2:} $x\in L$, $y\in U$ (or vice versa). Let $x\in L_i$ and $y=u_j$. Since $d(x)=4$ and $d(y)=2\ell/3\ge 8$, it is obvious that $N(\overline x,y)\ne\emptyset$ and $N(\overline x,\overline y)\ne\emptyset$. Moreover, if $k$ is chosen such that $k\in\{1,2,3\}\sm\{i,j\}$, then  $u_k\in N(x,\overline y)$. So, it remains to show $N(x,y)\ne\emptyset$. If $j=i$, then by Property (2) there exists a vertex $z\in N(x)\cap L_{m}$ for some $m\ne i$. Then $z\in N(x,y)$. If $j\ne i$, then by Property (2) there exists a vertex $z\in N(x)\cap L_i$. Then $z\in N(x,y)$.\\
    \textit{Case 3:} $x,y\in U$. Let $x=u_i$ and $y=u_j$. If $k$ is chosen such that $k\in\{1,2,3\}\sm\{i,j\}$, then any vertex $z\in L_k$ satisfies $z\in N(x,y)$, and the vertex $u_k$ satisfies $u_k\in N(\overline x,\overline y)$. Moreover, any vertex $z\in L_j$ satisfies $z\in N(x,\overline y)$ and any vertex $z\in L_i$ satisfies $z\in N(\overline x,y)$. 
\end{proof}

\begin{cor} \label{Corollary 6k+3}
    If $\ell\ge 12$ is a multiple of 6, then with $v=\ell+3$, we have $m(v,2)\le 3v-9$.
\end{cor}

We have proved the upper bound of \autoref{Theorem m(v,2)} for the case where $v\equiv 3\pmod 6$; now we extend the bound to all large enough $v$ by making small modifications to the graph $D_\ell$. We will use the operation of \textit{edge replication} defined in \cite{bonato2001adjacency} as follows. For a graph $G$ and an edge $e=xy\in E(G)$, let $R(G,e)$ be the graph with vertices $V(G)\sqcup\{x',y'\}$ and edges $E(G)\cup\{x'y'\}\cup\{x'z:xz\in E(G)\text{ and }z\ne y\}\cup\{y'z:yz\in E(G)\text{ and }z\ne x\}.$ If $G$ is 2-e.c. and $e\in E(G)$, then $R(G,e)$ is also 2-e.c. \cite{bonato2001adjacency}.

\begin{cor} \label{Corollary Odd v}
    If $v\ge 15$ is odd then $m(v,2)\le 3v+5.$
\end{cor}
\begin{proof}
    If $G$ is $2$-e.c., then replicating an edge results in a graph $G'$ with $v(G')=v(G)+2$ and which is $2$-e.c.. If there exists an edge in $G$ both of whose endpoints have degree at most $C$, then $e(G')-e(G)\le 2C-1$. Moreover, the graph $G'$ has an edge both of whose endpoints have degree at most $C$. For the graph $D_\ell$, we may take $C=4$. For any odd $v\not\equiv3\pmod 6$, there is some integer $u=6k+3$ such that $v-u\in\{2,4\}.$ Then applying one or two edge replication to $D_{6k}$ produces a $2$-e.c. graph on $v$ vertices with at most $3v-9+2\cdot(2\cdot 4-1)=3v+5$ edges.
\end{proof}

\begin{cor} \label{Corollary Even v}
    If $v\ge 20$ is even then $m(v,2)\le 3v+25$.
\end{cor}
\begin{proof}
    First we prove the claim where $v$ is of the form $6k+2$. Set $\ell=6(k-1)\ge 12$, and let $D_\ell'$ be the graph obtained from $D_\ell$ by adding 5 new vertices $W=\{w_1,\ldots,w_5\}$ such that $W$ induces a 5-cycle, $w_i\sim u_j$ for every $i,j$, and $w_i\not\sim x$ for all $i\in[5],x\in L$. Then $v(D_\ell')=v$. To show that $D_\ell'$ is 2-e.c., we need only check pairs $(x,y)$ where at least of of $x,y$ belongs to $W$. If $x,y\in W$, then $u_1\in N(x,y)$. Since $W$ induces a 5-cycle, the sets $N(\overline x,y)$ and $N(x,\overline y)$ are nonempty. We have that any $z\in L$ satisfies $z\in N(\overline x,\overline y)$. If $x\in W,y\in U$, then any element of $N_W(x)$ belongs to $N(x,y)$. Any element of $U\sm\{y\}$ belongs to $N(x,\overline y)$. Since $y$ has a neighbor $z_1\in L$ and a non-neighbor $z_2\in L$, we have $z_1\in N(\overline x,y)$ and $z_2\in N(\overline x,\overline y)$. Finally, if $x\in W,y\in L$, then any element of $N_U(y)$ belongs to $N(x,y)$. Any element of $N_W(x)$ belongs to $N(x,\overline y)$. Any element of $N_L(y)$ belongs to $N(\overline x,y)$. A non-neighbor of $y$ in $L$ belongs to $N(\overline x,\overline y)$. Thus, $D_\ell'$ is 2-e.c. and so $m(v,2)\le e(D_\ell')=3v-9+5+15=3v+11.$

    Now let $v\not\equiv 2\pmod 6$ be an even integer. There exists an integer $u=6k+2$ with $v-u\in\{2,4\}$. The graph $D_{6k}'$ contains an edge both of whose endpoints have degree 4, so applying edge-replication in similar fashion to Corollary \ref{Corollary Odd v} proves that $m(v,2)\le 3v+11+2(2\cdot 4-1)=3v+25.$
\end{proof}

\begin{cor} \label{m(v,2) asymptotic upper bound}
    We have $m(v,2)\le 3v+25$ for all $v\ge 20$.
\end{cor}

This completes the proof of the upper bound in \autoref{Theorem m(v,2)}.

\subsubsection{Lower bound}

We now prove the lower bound in \autoref{Theorem m(v,2)}, restated below.

\begin{prop} \label{Proposition m(v,2) lower bound}
    We have $m(v,2)\ge 3v-5625$ for all $v\ge 3201$.
\end{prop}

Our overall strategy is as follows. Let $G$ be an extremal 2-e.c. graph with $m(v,2)$ edges. We partition its vertices into those with high degree $H$, and those with low degree, $L$; $H$ cannot be too large, or else $e(G)$ exceeds the known upper bound on $m(v,2)$. Using the bound $e(G)\ge e(L,H)+e(L)=\sum_{\ell\in L}d_H(\ell)+\frac{1}{2}d_L(\ell)$, we are done if $|H|=O(1)$ and all vertices $\ell\in L$ satisfy $d_H(\ell)\ge 2$ and $d_L(\ell)\ge 2$ (as in the graph $D_\ell$). We can show that there are two vertices $h_1,h_2\in H$ such that all but $O(1)$ vertices are adjacent to at least one of $h_1,h_2$; but it is possible that $\omega(1)$ vertices are adjacent to only one of $h_1,h_2$. In this case (addressed below \autoref{Claim A1 A2 size}), we must find a different partition $V=X\sqcup Y$ for which $e(X,Y)+e(X)\ge 3v-O(1).$ We do this by an iterative process in which vertices adjacent to only one of $h_1,h_2$ have their neighbors added to the new `high-degree' set $Y$.

\begin{proof}[Proof of Proposition \ref{Proposition m(v,2) lower bound}]
    Let $G=(V,E)$ be a 2-e.c. graph with $e=m(v,2)$ edges and note that $e\le 3v+25$. Set $L=\{x\in V:d(x)< 43\}$, and set $H=V\sm L$. Then
    $$21.5|H|\le\frac{1}{2}\sum_{h\in H}d(h)\le e\le 3v+25$$
    and so $|H|\le\frac{3v+25}{21.5}\le 0.14v+2.$
    \begin{claim} \label{Claim 2 H neighbors}
    There exists $\ell_1\in L$ such that $d_H(\ell_1)\le 2$.
    \end{claim}
    \begin{proof}
        Suppose not. Then we have
        $$\begin{aligned}
            e&\ge e(L,H)+e(L)=\sum_{\ell\in L}d_H(\ell)+\frac{1}{2}\sum_{\ell\in L}d_L(\ell)
            =\sum_{\ell\in L}d_H(\ell)+\frac{1}{2}d_L(\ell).
        \end{aligned}$$
        Since $d_H(\ell)+d_L(\ell)\ge \delta(G)\ge 4$, we have $d_H(\ell)+\frac{1}{2}d_L(\ell)\ge 3+\frac{1}{2}\cdot 1=3.5.$ Thus,
        $$\sum_{\ell\in L}d_H(\ell)+\frac{1}{2}d_L(\ell)\ge 3.5|L|\ge 3.5(v-0.14v-2)=3.01v-7=3v-7+0.01v>3v+25,$$
        since $v\ge 3201$, a contradiction. 
    \end{proof}
    Let $\ell_1$ be as in Claim \ref{Claim 2 H neighbors}. Let $h_1,h_2$ be two  neighbors of $\ell_1$ such that $N(\ell_1)\sm\{h_1,h_2\}\subseteq L$. Set $A=(N(h_1)\cup N(h_2))\sm\{h_1,h_2\}.$
    \begin{claim} \label{Claim h1 h2 dominating}
        We have $|A|\ge v-1849.$
    \end{claim}
    \begin{proof}
    Let $x\in V\sm A$. Since $N(x,\ell_1)\ne\emptyset$, we have $x\sim \ell$ for some $\ell\in N(\ell_1)\sm\{h_1,h_2\}$. But the number of $x\in V$ with this property is at most
    $$\sum_{\ell\in N(\ell_1)\sm\{h_1,h_2\}}d(\ell)<\sum_{\ell\in N(\ell_1)\sm\{h_1,h_2\}}43<43\cdot 43=1849.$$
    \end{proof}
    Set $A_1=N(h_1,\overline{h_2})$, $A_2=N(h_2,\overline{h_1})$, and $A_{12}=N(h_1,h_2)$.
    \begin{claim} \label{Claim A1 A2 size}
        If $|A_1|\le 26$ or $|A_2|\le 26$, then $e\ge 3v-5625$.
    \end{claim}
    \begin{proof}
        By symmetry, we may assume that $|A_2|\le 26$. By Claim \ref{Claim h1 h2 dominating}, this implies that $d(h_1)\ge v-1875$. Set $N=N(h_1)$, and note that every vertex in $N$ has two neighbors outside $N$, one of which is $h_1$. Thus, for $x\in N$, we have $d_{V\sm N}(x)+d_N(x)\ge \delta(G)\ge 4$ and so $d_{V\sm N}(x)+\frac{1}{2}d_N(x)\ge 2+\frac{1}{2}\cdot 2=3$. Then
        $$e(G)\ge e(N,V\sm N)+e(N)=\sum_{x\in N}d_{V\sm N}(x)+\frac{1}{2}d_N(x)\ge 3|N|=3v-3\cdot 1875=3v-5625.$$
    \end{proof}
    By Claim \ref{Claim A1 A2 size}, we are done if $|A_1|\le 26$ or $|A_2|\le 26$, so assume that $|A_1|>26$ and $|A_2|>26$. If every element of $A_1\cup A_2$ has degree at least 5, then setting $K=\{h_1,h_2\}$, we have
    $$e(G)\ge\sum_{x\not\in K}d_K(x)+\frac{1}{2}d_{V\sm K}(x)\ge\sum_{x\in A_{12}}\left(2+\frac{1}{2}\cdot 2\right)+\sum_{x\in A_1\cup A_2}\left(1+\frac{1}{2}\cdot 4\right)=3|A|\ge 3v-5547,$$
    and we are done. Otherwise, assume there exists $x_1\in A_1$ such that $d(x_1)=4$. Set $A^{(1)}=A,$ $A_1^{(1)}=A_1$, $B_1^{(1)}=\{a\in A_1: d(a)=4\}$, $C_1^{(1)}=\{a\in A_1:d(a)\ge 5\}$ $A_2^{(1)}=A_2$, $A_{12}^{(1)}=A_{12}$. Set $N_1=N(x_1)\sm\{h_1\}$ and note that $|N_1|=3$. Now set $A^{(2)}=A\sm N_1$, $A_1^{(2)}=A_1^{(1)}\sm N_1$, $B_1^{(2)}=\{a\in A_1^{(2)}: d(a)=4\}$, $C_1^{(2)}=\{a\in A_1^{(2)}: d(a)\ge 5\}$, $A_2^{(2)}=A_2^{(1)}\sm N_1$, $A_{12}^{(2)}=A_{12}^{(1)}\sm N_1$. In general, suppose we have defined $A^{(i)}$, $A_1^{(i)}$, $B_1^{(i)}$, $C_1^{(i)}$, $A_2^{(i)}$, $A_{12}^{(i)}$, and $N_1,\ldots,N_{i}$. We stop unless there exists $x_{i+1}\in B_1^{(i)}$ with the property that $x_{i+1}$ has no neighbors in $N_1\cup\cdots\cup N_{i}$. In that case, set 
    $$\begin{aligned}
            N_{i+1}&=N(x_{i+1})\sm \{h_1\},\\ A^{(i+1)}&=A^{(i)}\sm N_{i+1},\\ A_1^{(i+1)}&=A_1^{(i)}\sm N_{i+1},\\ B_1^{(i+1)}&=\{a\in A_1^{(i+1)}:d(a)=4\},\\
            C_1^{(i+1)}&=\{a\in A_1^{(i+1)}:d(a)\ge 5\},\\
            A_2^{(i+1)}&=A_2^{(i)}\sm N_{i+1},\\
            A_{12}^{(i+1)}&=A_{12}^{(i)}\sm N_{i+1}
    \end{aligned}$$
 (see Figure \ref{Figure vertex partition}). This process eventually terminates, say after finding $x_k$ ($k\ge 1$). 

\begin{figure}
\begin{center}
\begin{tikzpicture}[scale=1]
    \draw[very thick] (-1.5,0)--(1,2)--(0,0);
    \draw[very thick] (1,2)--(3,0)--(4,2)--(6,0);
    \draw[fill=black] (1,2) circle (2pt) node[above] {$h_1$};
    \draw[fill=black] (4,2) circle (2pt) node[above] {$h_2$};
    \draw[fill=white] (0,0) ellipse (.5 and .5);
    \draw[fill=white] (-2,0) ellipse (1 and .5);
    \node at (-2,-.9) {$B_1^{(i)}$};
    \node at (0,-.9) {$C_1^{(i)}$};
    \draw[fill=white] (3,0) ellipse (1.7 and .6);
    \node at (3,-0.9) {$A_{12}^{(i)}$};
    \draw[fill=white] (6,0) ellipse (.5 and .5);
    \node at (6,-0.9) {$A_2^{(i)}$};
    \draw[fill=white] (8,0) ellipse (.5 and .5);
    \node at (8,-0.9) {$\le 1849$};
    \draw[fill=black] (-2.5,0.1) circle (2pt) node[below] {$x_1$};
    \node at (-2,0.1) {$\cdots$};
    \draw[fill=black] (-1.5,0.1) circle (2pt) node[below] {$x_{i}$};
    \draw[fill=white] (-3,2) ellipse (.5 and .2);
    \draw[fill=white] (-1,2) ellipse (.5 and .2);
    \node at (-3,2.4) {$N_1$};
    \node at (-1,2.4) {$N_{i}$};
    \draw[fill=black] (-3.25,2) circle (2pt);
    \draw[fill=black] (-3,2) circle (2pt);
    \draw[fill=black] (-2.75,2) circle (2pt);
    \draw[fill=black] (-1.25,2) circle (2pt);
    \draw[fill=black] (-1,2) circle (2pt);
    \draw[fill=black] (-0.75,2) circle (2pt);
    \draw (-3.25,2)--(-2.5,0.1)--(-3,2);
    \draw (-2.5,0.1)--(-2.75,2);
    \draw (-1.25,2)--(-1.5,0.1)--(-1,2);
    \draw (-1.5,0.1)--(-0.75,2);
\end{tikzpicture}
\caption{The partition of $V(G)$ obtained after step $i$.} \label{Figure vertex partition}
\end{center}
\end{figure}
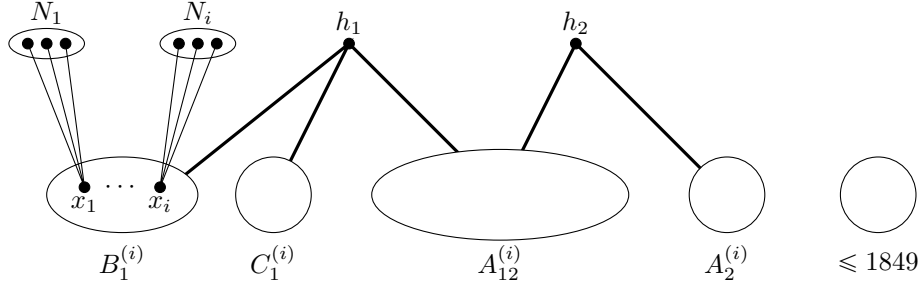
    
    Now set $N=N_1\sqcup\cdots\sqcup N_k$. We note the following facts:
    \begin{enumerate}
        \item $V=B_1^{(k)}\sqcup C_1^{(k)}\sqcup A_2^{(k)}\sqcup A_{12}^{(k)}\sqcup K\sqcup  N\sqcup D$, with $|K|=2,|N|=3k$, and $|D|\le1849$.
        \item Every vertex in $B_1^{(k)}$ has a neighbor in $K$ and a neighbor in $N$.
        \item Every vertex in $C_1^{(k)}$ has degree at least 5. 
        \item Every vertex in $A_2^{(k)}$ has a neighbor in each of $N_1,\ldots, N_k$. (This is because it has a neighbor in common with each of $x_1,\ldots,x_k$, and this common neighbor cannot be $h_1$.)
        \item Every vertex in $A_{12}^{(k)}$ has two neighbors in $K$.
    \end{enumerate}
    Set $X=A^{(k)}$ and $Y=V\sm A^{(k)}$. Note that $|X|+|N|=|A|\ge v-1849$. For any $x\in X\sm (C_1^{(k)}\cup A_2^{(k)})$, we have $d_Y(x)+\frac{1}{2}d_X(x)\ge 2+\frac{1}{2}\cdot 2=3$; for any $x\in C_1^{(k)}$ we have $d_Y(x)+\frac{1}{2}d_X(x)\ge 1+\frac{1}{2}\cdot 4=3$. Set $m=|A_2\cap N|$, so that $|A_2^{(k)}|=A_2-m$. Since every vertex in $A_2$ has a neighbor in each of $N_1,\ldots,N_k$, every vertex $y\in A_2$ satisfies $d_N(y)\ge k$ and thus $e(N)\ge km/2$. 

    \textit{Case 1:} $k\ge 3$. Then for $x\in A_2^{(k)}$, we have $d_Y(x)+\frac{1}{2}d_X(x)\ge d_Y(x)\ge k+1$.

    Then we have 
    $$\begin{aligned}
        e(G)&=e(X,Y)+e(X)+e(Y)\ge\sum_{x\in X}\left(d_Y(x)+\frac{1}{2}d_X(x)\right)\ge\sum_{x\in X\sm A_2^{(k)}}3+\sum_{x\in A_2^{(k)}}(k+1).
    \end{aligned}$$
Hence,
    $$\begin{aligned}
        e(G)&\ge 3(|X|-|A_2^{(k)}|)+(k+1)|A_2^{(k)}|+\frac{km}{2}\\
        &=3|X|+(k-2)|A_2^{(k)}|+\frac{km}{2}\\
        &=3|X|+(k-2)(|A_2|-m)+\frac{km}{2}.\\
    \end{aligned}$$
\textit{Subcase 1.1:} $m\ge 18$. Then $\frac{km}{2}\ge 9k=3|N|$, and we have
$$3|X|+(k-2)(|A_2|-m)+\frac{km}{2}\ge 3|X|+\frac{km}{2}\ge 3(|X|+|N|)\ge 3v-5547.$$
We are done.

\textit{Subcase 1.2:} $m<18$. Since $|A_2|>26$, we have $|A_2|-m\ge 27-18= 9$. Then
$$e(G)\ge 3|X|+9(k-2)+\frac{km}{2}\ge 3|X|+9k-18=3|X|+3|N|-18=3v-5565$$
and we are done.

\textit{Case 2:} $k\le 2$. Then for $x\in A_2^{(k)}$, we have $d_Y(x)+\frac{1}{2}d_X(x)=(k+1)+\frac{4-(k+1)}{2}\ge 3$, so $d_Y(x)+\frac{1}{2}d_X(x)\ge 3$ holds for all $x\in X$. We have
$$e(G)\ge e(X,Y)+e(X)=\sum_{x\in X}d_Y(x)+\frac{1}{2}d_X(x)\ge 3|X|.$$
Now $|N|=3k\le 6$, and so $3|X|=3(|A|-|N|)\ge 3(v-1849-6)=3v-5565)$ are we are done.
\end{proof}

\subsection{The case $n\ge 3$}

In contrast to the situation with $n=2$, we will show that if $n\ge 3$, then a $v$-vertex $n$-e.c. graph has $\omega(v)$ many edges. Lower bounds on $m(v,n)$ will follow from the connection between $n$-e.c. graphs and $n$\textit{-strong binary covering arrays}. Following the notation of \cite{lawrence2011survey}, we define a \textit{binary $n\times k$ covering array of strength $t$}, or $\mathrm{CA}(n,k,t)$, to be an $n\times k$ $\{0,1\}$-matrix with the property that every $n\times t$ submatrix contains every binary vector of length $t$ as some row. As noted in \cite{colbourn2009binary}, if $G$ is a $v$-vertex $n$-e.c. graph, then the adjacency matrix of $G$ is a $\mathrm{CA}(v,v,n)$. Kleitman and Spencer \cite{kleitman1973families} proved the following bound on the dimensions of a $\mathrm{CA}(n,k,t)$.

\begin{thm}[Kleitman-Spencer \cite{kleitman1973families}] \label{Theorem Kleitman Spencer} Suppose that a $\mathrm{CA}(n,k,t)$ exists. Then $n\ge (c_t+o(1))\log_2k$, where $c_t$ is a constant satisfying
$$c_t\ge\frac{t-2}{(t-1){2^{1-t}}-(1-2^{1-t})\log_2(1-2^{1-t})-2^{2-t}}\sim 2^{t-1}$$
if $t\ge 3$, and $c_t=1$ if $t=2$.
\end{thm}

\begin{thm} \label{Theorem n>2 lower bound}
    If $n\ge 3$, we have $m(v,n)\ge (c_{n-1}+o(1))v\log_2 v.$
\end{thm}
\begin{proof}
    Suppose that $G=(V,E)$ is an $n$-e.c. graph with minimum degree $\delta$. Let $x$ be a vertex of degree $\delta$. Let $A$ be the adjacency matrix of $G$ and let $B$ be the submatrix of $A$ consisting of the rows indexed by vertices in $N(x)$. We claim that $B$ is a $\mathrm{CA}(\delta,v,n-1)$. Let $x_1,\ldots,x_k,x_{k+1},\ldots,x_{n-1}$ be $n-1$ distinct vertices in $V$. Since $G$ is $n$-e.c., we have $N(x,x_1,\ldots,x_k,\overline{x_{k+1}},\ldots,\overline{x_{n-1}})\ne\emptyset$, so let $y$ be an element of this set. Then the row of $A$ corresponding to $y$ is also a row of $B$, and it contains 1's in the columns corresponding to $x_1,\ldots,x_k$ and 0's in the columns corresponding to $x_{k+1},\ldots,x_{n-1}$. Since the choice of $x_1,\ldots,x_n$ and $k$ were arbitrary, every binary vector of length $n-1$ can be obtained in this way, proving the claim. Since $n-1\ge 2$, Theorem \ref{Theorem Kleitman Spencer}, gives $\delta\ge (c_{n-1}+o(1))\log_2v$, and the result follows.
\end{proof}

For upper bounds, we are not able to rely directly on the connection to covering arrays. The natural correspondence would be as follows. Let $B$ be a $k\times v$ $n$-strong covering array. Then define a graph on vertex set $V=K\sqcup U$, where $K$ is a set of size $k$ and $U$ is a set of size $v-k$. We identify $K$ with the rows of $B$ and $V$ with the columns of $B$. A vertex $v$ is adjacent to the vertices in $K$ corresponding to the 1's in the column corresponding to $v$ (and so $U$ is an independent set). Unfortunately, this construction may yield directed edges within $K$. We can avoid this if $B$ contains a symmetric $k\times k$ submatrix, but it is not clear to us that this is necessarily the case. (It may be an interesting problem to find conditions under which $B$ necessarily contains a symmetric $k\times k$ submatrix.) Alternatively, we could delete the edges within $K$ and replace them by a random graph. However, this may break the $n$-e.c. condition for sets $x_1,\ldots,x_n$ which intersect both $K$ and $U$. 

To prove an upper bound on $m(v,n)$, we will apply the Lov\'asz Local Lemma directly to a random graph construction, rather than going through the intermediate step of obtaining a covering array. (The Local Lemma was used to obtain bounds on the parameters of a $\mathrm{CA}(n,k,t)$ in \cite{Godbole_Skipper_Sunley_1996}.)

\begin{thm}[Erd\H{o}s-Lov\'asz 1975 \cite{erdHos1975problems}] \label{Theorem LLL}
    Let $E_1,\ldots,E_n$ be events in a probability space. Suppose that $\mathbb P[E_i]\le p$ for all $i$, and each $E_i$ is mutually independent of all but at most $d$ other events $E_j$. If $ep(d+1)\le 1$, then $\mathbb P[\bigcap_{i=1}^nE_i]>0$.
\end{thm}

\begin{thm} \label{n>2 upper bound}
    For every $n\ge 3$, we have $m(v,n)\le[(1+2^{(3-n)/2})(n-1)2^{n-1}+o(1)]v\log v.$
\end{thm}
\begin{proof}

For $v,t\in\mathbb N$ with $t\le v$, we define a random graph $G=G_{v,t}$ as follows. Let $G=(V,E)$ with $V=X\sqcup Y$, with $|X|=t$ and $|Y|=v-t$. For $x\in X$ and $z\in V\sm\{x\}$, we set $\mathbb P[xz\in E]=\frac{1}{2}$ with distinct edges being independent. Thus, $G$ is a random subgraph of the join of a clique on $X$ with an independent set on $Y$.

Fix $\delta>0$. We will show that, with positive probability, $G$ is $n$-e.c. and $e(G)\le(1/2+2^{(1-n)/2})vt.$ For $S\subseteq V$ with $|S|=n$ and $A\sqcup B=S$, define the event
$$E_{S,A,B}=\{N(A,\overline B)\cap X\ne \emptyset\}.$$
Furthermore, for $\varepsilon>0$ and $y\in Y$, define the event
$$E_y=\{d_Y(x)>(1+\varepsilon)t/2\}.$$
We bound the probabilities of the events $\overline{E_{S,A,B}}$ and $\overline{E_y}$. For $(S,A,B)$, note that there exist at least $t-n$ vertices in $X\sm S$, and all the edges between $S$ and $X\sm S$ are randomly chosen with probability $\frac{1}{2}$. Thus,
$$\mathbb P\left[\overline{E_{S,A,B}}\right]\le\left(1-2^{-n}\right)^{t-n}\le e^{-2^{-n}(t-n)}.$$
For $x\in X$, note that $d_Y(x)\sim \mathrm{Bin}(t,1/2)$. The Chernoff bound \cite{chernoff1952measure} gives
$$\mathbb P\left[\overline{E_y}\right]\le e^{-\varepsilon^2 (t/2)/(2+\varepsilon)}.$$
If we choose $\varepsilon=2^{(3-n)/2}$, then we have $e^{-\varepsilon^2(t/2)/(2+\varepsilon)}\le e^{-2^{-n}(t-n)}$. Thus, we set $p=e^{-2^{-n}(t-n)}$. 

We now construct a dependency digraph for the events in $\mathcal E$. Consider an event of the form $E_{S,A,B}$. The event $E_{S,A,B}$ is mutually independent of all events in $\mathcal E_{S,A,B}=\{E_{S',A',B'}:S'\subseteq Y,S'\cap S=\emptyset,A'\sqcup B'=S'\}$; for the events in $\mathcal E_{S,A,B}$ depend only on the edges in $E(X,S')$, which is disjoint from $E(X,S)$. It follows that $E_{S,A,B}$ is mutually independent of all but 
\begin{equation} \label{Equation LLL d}
\left[\binom{v-t}{n}-\binom{v-t-n}{n}\right]2^n+\left[\binom{v}{n}-\binom{v-t}{n}\right]2^n+v-t
\end{equation}
events in $\mathcal E$. Now consider an event of the form $E_y$. Since $y\in Y$, $E_y$ is mutually independent of all events in $\mathcal E_x=\{E_{S,A,B}:x\not\in S,A\sqcup B=S\}.$ Obviously $|\mathcal E_x|\ge|\mathcal E_{S,A,B}|$, so we set
$d$ equal to the expression in Equation \ref{Equation LLL d}. We estimate $d$ as follows:
$$\begin{aligned}
    d&=\left(\frac{n^2v^{n-1}+O(v^{n-2})}{n!}\right)2^n+\left(\frac{ntv^{n-1}+O(v^{n-2})}{n!}\right)2^n+v-t\\
    &=\left(\frac{n^2+nt}{n!}\right)2^nv^{n-1}+O(v^{n-2})=(2^n+o(1))\frac{tv^{n-1}}{(n-1)!}.
\end{aligned}$$
To apply Theorem \ref{Theorem LLL} we require
$$(2^n+o(1))ee^{-2^{-n}(t-n)}\frac{tv^{n-1}}{(n-1)!}<1$$
which holds if we choose $t=\lceil(2^n(n-1)+\delta)\log v\rceil$. So, with positive probability, none of the events in $\mathcal E$ occur. This implies that $G$ is $n$-e.c. and
$$\begin{aligned}e(G)&\le e(X)+e(X,Y)\le\binom{t}{2}+t(1/2+\varepsilon/2)(v-t)\\
&\le t^2+\left(\frac{1}{2}+2^{(1-n)/2}\right)(2^n(n-1)+2\delta)v\log\le \left[\left(\frac{1}{2}+2^{(1-n)/2}\right)2^n(n-1)+3\delta\right] v\log v.\end{aligned}$$
Taking $\delta\to 0$ proves the result.
\end{proof}

Comparing Theorems \ref{Theorem n>2 lower bound} and \autoref{n>2 upper bound}, we see that the lower and upper bounds on $m(v,n)$ differ by a $\Theta(n)$ factor, as $n\to\infty$.

\section{VC-dimension}

For a graph $G$, define the \textit{VC-dimension} of $G$ to be the VC-dimension of the set system $\{N(u):u\in V(G)\}$. Some authors have also considered the closed-neighborhood set system $\{N[u]:u\in V(G)\}$, but we will restrict our attention to the open-neighborhood version. In this section, we investigate the VC-dimension of strongly regular graphs. First, we consider general lower and upper bounds, before investigating the VC-dimension of specific families of SRGs.

We note that, in contrast with $n$-e.c., the VC-dimension of $G$ is not necessarily equal to the VC-dimension of $\overline G$, since $N_{\overline G}(x)\ne V(G)\sm N_G(x)$. This occasionally leads to nearly identical proofs involving a graph and its complement, which nevertheless cannot be condensed to a single case.

Our focus is on the VC-dimension of strongly regular graphs, defined as follows.

\begin{defn} Let $G$ be a graph on $v$ vertices. We say that $G$ is \textit{strongly regular} if there exist integers $k,\lambda,\mu$ such that
\begin{enumerate}
    \item $d(x)=k$ for all $x\in V$;
    \item $d(x,y)=\lambda$ for all $x,y\in V$ with $x\sim y$;
    \item $d(x,y)=\mu$ for all distinct $x,y\in V$ with $x\not\sim y$.
\end{enumerate}
We say that $G$ is an $\mathrm{SRG}(v,k,\lambda,\mu)$. If either $G$ or $\overline{G}$ is disconnected, we say that $G$ is \textit{imprimitive}; otherwise, it is \textit{primitive.}
\end{defn}

\subsection{General bounds}

\begin{lemma}[see \cite{brouwer2022strongly}, Section 1.1.1] \label{Lemma SRG equation}
If $G$ is an $\mathrm{SRG}(v,k,\lambda,\mu)$, then $k(k-\lambda-1)=\mu(v-1-k).$
\end{lemma}

\begin{lemma}[see \cite{brouwer2022strongly}, Section 1.1.2] \label{Lemma SRG complement}
    If $G$ is an $\mathrm{SRG}(v,k,\lambda,\mu)$, then $\overline G$ is an $\mathrm{SRG}(v,\overline k,\overline\lambda,\overline\mu)$ with
    $$\begin{aligned}
        \overline k&=v-k-1,\\
        \overline \lambda&=v-2k+\mu-2,\\
        \overline\mu&=v-2k+\lambda.
    \end{aligned}$$
\end{lemma}

\begin{lemma}[see \cite{brouwer2022strongly}, Section 1.1.3] \label{Lemma SRG positive}
    Suppose that $G$ is a primitive $\mathrm{SRG}(v,k,\lambda,\mu)$. Then the quantities $\mu,\overline\mu,k-\lambda-1,\overline k-\overline\lambda-1,k-\mu$, and $\overline k-\overline\mu$ are all positive.
\end{lemma}

\begin{prop} \label{Proposition primitive SRG 2}
    If $G$ is a primitive SRG, then the VC-dimension of $G$ is at least 2.
\end{prop}
\begin{proof}
    Since $G$ is primitive, its parameters $v,k,\lambda,\mu$ satisfy certain constraints. First, since $G$ is connected and not a complete graph, we have $\mu>0$. Second, since $\overline G$ is also primitive, we must have $\overline \lambda<\overline k-1$; but $\overline\lambda = v-2k+\mu-2$ and $\overline k= v-k-1$, so this implies $\mu<k$.

    Now choose $x,y\in V(G)$ such that $x\not\sim y$. Since $\mu >0$, there is some $z\in V(G)$ with $\{x,y\}\subseteq N(z)$. Since $\mu<k$, there is some $w\in N(x)\sm N(y)$, so that $N(w)\cap\{x,y\}=\{x\}$; similarly there is some $w'$ with $N(w')\cap\{x,y\}=\{y\}$. Finally, we have $N(x)\cap\{x,y\}=\emptyset$ (see Figure \ref{Figure SRG}).
\end{proof}
\begin{figure}[h]
\begin{center}
\begin{tikzpicture}[scale=0.95]
    \draw[fill=black] (0,2) circle (2pt) node[above] {$x$};
    \draw[fill=black] (0,0) circle (2pt) node[below] {$y$};
    \draw[dashed, thick] (2,1) -- (0,0) -- (-3,1) -- (0,2) -- (-.9,1);
    \draw[thick] (-.9,1) -- (0,0) -- (.5,1) -- (0,2) -- (2,1);

    \draw[fill=white] (.5,1) ellipse (.4 and .4) node[black] {$\mu$};
    \draw[fill=white] (-.9,1) ellipse (.6 and .4) node[black] {$k-\mu$};
    \draw[fill=white] (2,1) ellipse (.6 and .4) node[black] {$k-\mu$};
    \draw[fill=white] (-3,1) ellipse (1.3 and .4) node[black] {$v-2k+\mu-2$};
    \draw [dashed] (0,0)--(0,2);
\end{tikzpicture}
\end{center}
\caption{In a primitive SRG, the nonadjacent vertices $x,y$ form a shattered set.} \label{Figure SRG}
\end{figure}
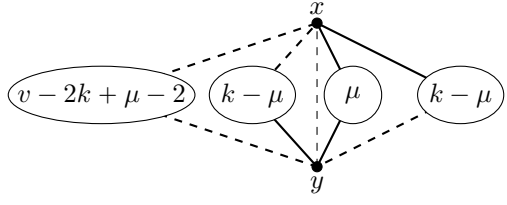

In the rest of this subsection, we will study the question of when the VC-dimension of a primitive SRG is equal to 2. The next two propositions give some examples where this holds.

\begin{prop} \label{Proposition C4 free}
    Suppose that $G$ is a $C_4$-free graph. Then the VC-dimension of $G$ is at most 2.
\end{prop}
\begin{proof}
    Assume, for a contradiction, then $x,y,z$ are 3 distinct vertices of $G$ such that $\{x,y,z\}$ is shattered by $\{N(u):u\in V(P)\}.$ Then $x,y,z$ must have a common neighbor, so $\{x,y,z\}\subseteq N(u)$ for some $u$. Also, there must exist $w\in V(G)$ such that $w\in N(x,y,\overline z)$. Note that $w\not\in\{u,x,y\}$. Thus, $uxwyu$ is a 4-cycle in $G$, a contradiction. 
\end{proof}

\autoref{Proposition C4 free} tells us that $C_5$, the Petersen graph, and the Hoffman-Singleton graph each have VC-dimension 2, and if a Moore graph of order 3250 exists, then it must also have VC-dimension 2. No other primitive SRGs are $C_4$-free (see e.g. \cite{deutsch2001strongly}).

\begin{prop}\label{Proposition mu bar = 1}
    Suppose that $G$ is a primitive $\mathrm{SRG}(v,k,\lambda,\mu)$ with $\overline\mu=1$. Then the VC-dimension of $G$ is 2.
\end{prop}
\begin{proof}
    By Proposition \ref{Proposition primitive SRG 2}, we only need to prove the upper bound. So assume, for a contradiction, that there is a shattered set $X=\{x_1,x_2,x_3\}$ of size 3. Note that $X\subseteq N'(\overline z)$ for some $z\in V$. Now $G[N'(\overline z)]\simeq K_{\overline\lambda+1,\ldots,\overline\lambda+1}$. If $x_1\sim x_2$, then $z$ is the only vertex nonadjacent to both $x_1$ and $x_2$, so $\{x_1,x_2\}\not\subseteq X$. Thus, $X$ is a $\overline{K_3}$, and $X$ is contained in an independent set of size $\overline\lambda+2$. Now if $y_1,y_2,y_3\in X$ then $N'(\overline y_1,\overline y_2)$ is precisely this independent set, so $N'(\overline{y_1},\overline{y_2},y_3)=\emptyset$ and $X$ is not shattered.
\end{proof}

We note that there are simple examples of general graphs $G$ where $\overline G$ is $C_4$-free and yet the VC-dimension of $G$ is 3. The graphs in \autoref{Proposition mu bar = 1} are the complements of primitive geodetic SRGs. The only known examples of primitive geodetic SRGs are the three aforementioned Moore graphs of diameter 2 \cite{blokhuis1988geodetic}. In the rest of this section, we will show that for sufficiently large $v$, the examples in Propositions \ref{Proposition C4 free} and \ref{Proposition mu bar = 1} are the only primitive $v$-vertex SRGs with VC-dimension 2. We will use the following two results on possible parameter sets.

\begin{thm}[Hoffman-Singleton \cite{hoffman1960moore}] \label{Theorem Hoffman-Singleton}
    If $G$ is a primitive $\mathrm{SRG}(v,k,0,1)$, then $v\in\{5,10,50,3250\}.$
\end{thm}

\begin{lemma}[see \cite{deutsch2001strongly}] \label{Lemma vk11}
    A primitive $\mathrm{SRG}(v,k,1,1)$ does not exist.
\end{lemma}

Our basic strategy is to search for a $K_3$ or a $\overline{K_3}$ in $G$. This is because the symmetry in an SRG makes such subgraphs good candidates to be shattered, in the following way.

\begin{lemma} \label{Lemma triangle}
    Suppose that $G$ is an SRG which contains a $K_3$ with vertex set $X$. If there exist $r,s,t\in V(G)$ such that $d_X(r)=0$, $d_X(s)=1$, and $d_X(t)=3$, then $X$ is shattered.
\end{lemma}
\begin{proof}
    Let $X=\{x_1,x_2,x_3\}$. By assumption, $N_X(r)=\emptyset$ and $N_X(t)=X$. Since $G$ is an SRG, we have that $d'(x_i,\overline{x_j},\overline{x_k})=k-2\lambda+d(X)$. However, since $d_X(s)=1$, this quantity is positive when $x_i\sim s$, so it is positive for all orderings $i,j,k$. Finally, we have $N_X(x_i)=\{x_j,x_k\}$. 
\end{proof}

\begin{lemma} \label{Lemma anti triangle}
    Suppose that $G$ is an SRG which contains a $\overline{K_3}$ with vertex set $X$. If there exist $r,s,t\in V(G)$ such that $d_X(r)=1$, $d_X(s)=2$, and $d_X(t)=3$, then $X$ is shattered.
\end{lemma}
\begin{proof}
    Let $X=\{x_1,x_2,x_3\}$. We have $N_X(x_1)=\emptyset$. Now we have $d'(x_i,\overline{x_j},\overline{x_k})=k-2\mu+d(X).$ Since $d_X(r)=1$, this quantity is positive when $x_i\sim r$, so it is positive for all orderings $i,j,k$. We have $d'(x_i,x_j,\overline{x_k})=\mu-d(X)$. Since $d_X(s)=2$, this quantity is positive when $x_i,x_j\sim s$, so it is positive for all orderings $i,j,k$. Finally, $N_X(t)=X$. 
\end{proof}

Call a $K_3$ or $\overline{K_3}$ \textit{good} (with respect to $(r,s,t)$) if it satisfies the assumptions of Lemma \ref{Lemma triangle} or \ref{Lemma anti triangle}. To find a good $K_3$/$\overline{K_3}$, we will use Ramsey's theorem to find a large $K_t$/$\overline{K_t}$ so that we have many triples of vertices to choose from. Moreover, we want this $K_t$/$\overline{K_t}$ to live inside a set of the form $N(x,\overline y)$. The next lemma ensures that we can find a large enough $N(x,\overline y)$.

\begin{lemma} \label{Lemma Ramsey setup}
    Let $C>0$. If $v>(C^2-1)^2$, then every primitive $\mathrm{SRG}(v,k,\lambda,\mu)$ satisfies $k-\lambda\ge C$ or $k-\mu\ge C$.
\end{lemma}
\begin{proof}
    Suppose not. Then $k-\lambda<C$ and $k-\mu<C$. From Lemma \ref{Lemma SRG complement} and \ref{Lemma SRG equation} we obtain
    $$\overline k=v-k-1=\frac{k(k-\lambda-1)}{\mu}<\frac{(\mu+C)(C-1)}{\mu}=\left(1+\frac{C}{\mu}\right)(C-1)\le(C+1)(C-1).$$
    Since $\overline\mu>0$, applying Lemma \ref{Lemma SRG equation} to $\overline G$ gives
    $$\begin{aligned}v=\frac{\overline k(\overline k-\overline\lambda-1)}{\overline\mu}+\overline k+1&<\frac{(C+1)(C-1)((C+1)(C-1)-1)}{1}+C^2-1+1\\
    &=(C^2-1)(C^2-2)+(C^2-1)+1=(C^2-1)^2+1.\\\end{aligned}$$
\end{proof}

We are now ready to prove our main result. Once we find our large $K_t/\overline{K_t}$ contained in some $N(x,\overline y)$, we will show that if this $K_t/\overline{K_t}$ does not contain a good $K_3/\overline{K_3}$, then $G$ must have a relatively simple structure, which either makes it easy to find a shattered set or forces primitivity.

\begin{thm} \label{Theorem VC=2}
    Let $G$ be a primitive $\mathrm{SRG}(v,k,\lambda,\mu)$. If $v\ge 6401$, then the VC-dimension of $G$ is 2 if and only if $\overline\mu=1$.
\end{thm}

\begin{proof}
    It only remains to prove the ``only if" part. So, assume that $G$ is a primitive $\mathrm{SRG}(v,k,\lambda,\mu)$ with $v\ge6401$ and $\overline\mu\ne 1$; we will show $G$ contains a shattered set of size 3. Set $t=4$. The Ramsey number $R(4,3)$ is equal to $9$ \cite{greenwood1955combinatorial}, meaning that any graph on 9 vertices contains a $K_4$ or a $\overline{K_3}$. Note that $v\ge6401=(9^2-1)^2+1$, so by Lemma \ref{Lemma Ramsey setup}, there exist $x,y\in V$ with $d(x,\overline y)\ge 9$. Thus, there exists a clique $K$ of order $t=4$ or an independent set $I$ of order $t-1=3$, which in either case is contained is $N(x,\overline y)$. 

    \textit{Case 1:} there is a clique $K$ of order $t$ in $N(x,\overline y)$.

    \textit{Subcase 1.1:} there exists $z\in V$ such that $1\le d_K(z)\le t-2$. Then choose $x_1\in N_K(z)$ and $x_2,x_3\in K\sm N(z)$. The $K_3$ $\{x_1,x_2,x_3\}$ is good with respect to $(y,z,x)$.

    \textit{Subcase 1.2:} for all $z\in V$, $d_K(z)\in\{0,t-1,t\}$. For $i\in\{0,t-1,t\}$, set
    $$V_i=\{z\in V\sm K:d_K(z)=i\}$$
    and note that $V_0,V_t\ne\emptyset$. For all $u_1,u_2\in K$, we have $N(\overline{u_1},\overline{u_2})=V_0$, so $|V_0|=d(\overline{u_1},\overline{u_2})=\overline{\mu}$. For each $z\in V_{t-1}$, define $f(z)$ to be the unique $u\in K$ such that $z\not\sim u$. For each $u\in K$, we have $|f^{-1}(\{u\})|=\overline k-\overline\mu>0$. In particular, $V_{t-1}\ne\emptyset$ (see Figure \ref{Figure Subcase 1.2}).
    
\begin{figure} 
\begin{center}
\begin{tikzpicture}[scale=1.5]
    \draw[dashed, very thick] (1,0)--(2.5,1);
    \draw[very thick] (5.5,1)--(7,0);
    \draw[fill=lightgray] (4,1) ellipse (2 and .5) node[above] {$t$};
    \node at (4,1.7) {$K$}; 
    \draw[fill=white] (1,0) ellipse (.5 and .5) node {$\overline\mu$};
    \node at (1,-0.7) {$V_0$};
    \draw[fill=white] (4,-0.2) ellipse (2 and .5);
    \node at (4,-1) {$V_{t-1}$};
    \draw[dashed, very thick] (3.2,0.8)--(3.2,0);
    \draw[dashed, very thick] (4.8,0.8)--(4.8,0);
    \draw[very thick] (3.2,0.8)--(4.8,0);
    \draw[very thick] (3.2,0)--(4.8,0.8);
    \draw[fill=black] (3.2,0.8) circle (2pt)
    node[above] {$u_1$};
    \draw[fill=black] (4.8,0.8) circle (2pt)
    node[above] {$u_2$};
    \draw[fill=white] (3.2,0) ellipse (.4 and .2) node {\tiny $\overline k-\overline \mu$};
    \node at (3.2,-0.4) {$f^{-1}(\{u_1\})$};
    \draw[fill=white] (4.8, 0) ellipse (.4 and .2) node {\tiny $\overline k-\overline \mu$};
    \node at (4.8,-0.4) {$f^{-1}(\{u_2\})$};
    \draw[fill=white] (7,0) ellipse (.5 and .5);
    \node at (7,-0.7) {$V_t$};
\end{tikzpicture}
\caption{The structure of $G$ in Subcase 1.2. Here, $u_1$ and $u_2$ are arbitrarily chosen vertices in $K$.} \label{Figure Subcase 1.2}
\end{center}
\end{figure}
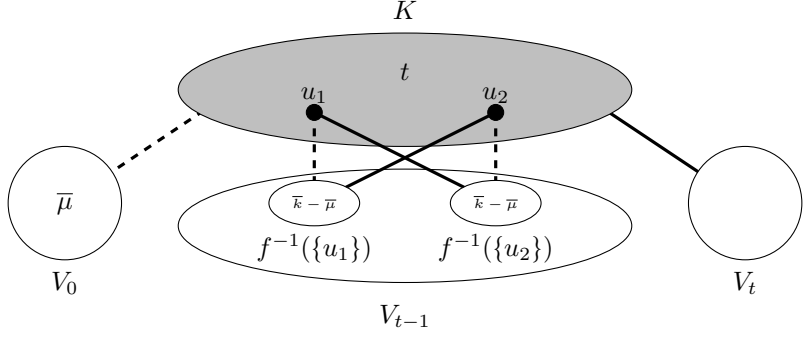

    \textit{Subcase 1.2.1:} there exist $z\in V_{t-1}$ and $w_1,w_2\in V_0$ such that $z\sim w_1$ and $z\not\sim w_2$. Since $t\ge 4$, we can choose $u_1,u_2,u_3\in K\sm\{f(z)\}$. The $K_3$ $\{z,u_1,u_2\}$ is good with respect to $(w_2,w_1,u_3)$.

    \textit{Subcase 1.2.2:} for all $z\in V_{t-1}$, we have $d_{V_0}(z)\in\{0,\overline\mu\}$. For $i\in\{0,\overline\mu\}$, set
    $$V_{t-1,i}=\{z\in V_{t-1}:d_{V_0}(z)=i\}.$$
    Now we claim that, for all $u\in K$, $f^{-1}(\{u\})\cap V_{t-1,\overline\mu}\ne\emptyset$. Suppose to the contrary that $f^{-1}(\{u\})\subseteq V_{t-1,0}$. Then for any $z_1,z_2\in f^{-1}(\{u\})$, we have $N(\overline{z_1},\overline{z_2})\supseteq V_0\cup\{u\}$, and so $d(\overline{z_1},\overline{z_2})\ge\overline{\mu}+1$ implying $z_1\not\sim z_2$. Thus, $f^{-1}(\{u\})$ is an independent set, and we have $\overline\lambda=d(\overline{z_1},\overline u)=\overline\mu+(\overline k-\overline\mu-1)=\overline k-1$, contradicting primitivity. 

    \textit{Subcase 1.2.2.1:} there exist $w_1,w_2\in V_0$ with $w_1\not\sim w_2$. Choose $u_1,u_2\in K$, and $z_1\in f^{-1}(\{u_1\})\cap V_{t-1,\overline\mu}$, $z_2\in f^{-1}(\{u_2\})\cap V_{t-1,\overline\mu}$. The $\overline{K_3}$ $\{w_1,w_2,u_1\}$ is good with respect to $(u_2,z_1,z_2)$.

    \textit{Subcase 1.2.2.2:} $V_0$ is a clique. Since $\overline\mu\ge 2$, there exist $w_1,w_2\in V_0$ with $w_1\sim w_2$.
    
    \textit{Subcase 1.2.2.2.1:} $V_{t-1,\overline{\mu}}$ is an independent set. Then choose $y'\in V_{t-1,\overline{\mu}}$, $x'=f(y')$, and $I=V_{t-1,\overline{\mu}}\sm f^{-1}(\{x'\})$, so that $|I|\ge t-1$. This is the situation of Case 2 (with $x,y$ replaced by $x',y'$), which is proved independently below.

    \textit{Subcase 1.2.2.2.2:} there exist $z_1,z_2\in V_{t-1,\overline\mu}$ such that $z_1\sim z_2$. Choose $u_1\in K\sm\{f(z_1)\}$. Then the $K_3$ $\{w_1,w_2,z_1\}$ is good with respect to $(f(z_1),u_1,z_2)$.

    \textit{Case 2:} there is an independent set $I$ of size $s=t-1$ in $N(x,\overline y)$.

    \textit{Subcase 2.1:} there exists $z\in V$ with $2\le d_I(z)\le s-1$. Set $I_z=N_I(z)$ and $I_{\overline{z}}=I\sm I_z$, and note that $|I_z|\ge 2$ and $|I_{\overline z}|\ge 1$.

    \textit{Subcase 2.1.1:} there exists $z_1\in V$ such that $d_I(z_1)>0$ and $|I_z\sm N(z_1)|\ge 2$. 
    
    \textit{Subcase 2.1.1.1:} there exists $u_1\in I_{\overline z}\cap N(z_1)$. Then choose $u_2,u_3\in I_z\sm N(z_1)$; the $\overline{K_3}$ $\{u_1,u_2,u_3\}$ is good with respect to $(z_1,z,x)$.

    \textit{Subcase 2.1.1.2:} $d_{I_{\overline{z}}}(z_1)=0$. Then there exists $u_1\in N_{I_z}(z_1)$; choose $u_2\in I_z\sm N(z_1)$ and $u_3\in I_{\overline z}$. The $\overline{K_3}$ $u_1,u_2,u_3$ is good with respect to $(z_1,z,x)$.

    \textit{Subcase 2.1.2:} for all $w\in V$, either $d_I(w)=0$ or $|I_z\sm N(w)|\le 1$. Set
    $$\begin{aligned}
        V_\emptyset&=\{w\in V\sm I:d_I(w)=0\}\\
        V_{0}&=\{w\in V\sm (I\cup V_{\emptyset}):|I_z\sm N(w)|=0\},\\
        V_1&=\{w\in V\sm(I\cup V_{\emptyset}):|I_z\sm N(w)|=1\}.
    \end{aligned}$$
    Note that $V_\emptyset,V_0\ne\emptyset$, and $V=I\sqcup V_\emptyset\sqcup V_0\sqcup V_1$ (see Figure \ref{Figure Subcase 2.1.2})

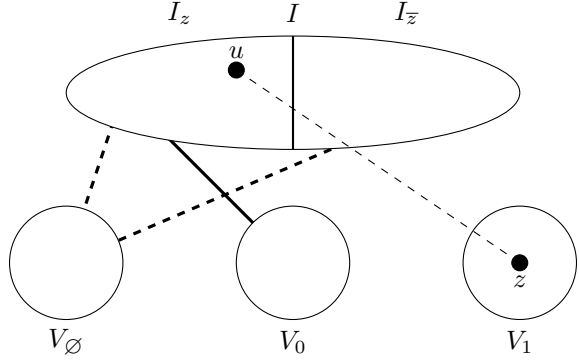
\begin{figure}
    \begin{center}
\begin{tikzpicture}[scale=1.5]
    \draw[dashed, very thick] (5.5,1)--(2,-0.5)--(2.5,1);
    \draw[very thick] (4,-0.5)--(2.5,1);
    \draw[fill=white] (4,1) ellipse (2 and .5);
    \draw[thick] (4,1.5)--(4,0.5);
    \node at (4,1.7) {$I$};
    \node at (3,1.7) {$I_z$};
    \node at (5,1.7) {$I_{\overline z}$};
    \draw[fill=white] (2,-0.5) ellipse (.5 and .5);
    \node at (2,-1.2) {$V_\emptyset$};
    \draw[fill=white] (4,-0.5) ellipse (.5 and .5);
    \node at (4,-1.2) {$V_{0}$};
    \draw[fill=white] (6,-0.5) ellipse (.5 and .5);
    \node at (6,-1.2) {$V_1$};
    \draw[fill=black] (3.5,1.2) circle (2pt);
    \draw[fill=black] (6,-0.5) circle (2pt);
    \node at (6,-.66) {$z$};
    \node at (3.5, 1.36) {$u$};
    \draw[dashed] (6,-0.5)--(3.5,1.2);
\end{tikzpicture}
\caption{The structure of $G$ in Subcase 2.1.2. Here, $z$ is an arbitrarily chosen vertex in $V_1$, and $u$ is its unique non-neighbor in $I_z$.} \label{Figure Subcase 2.1.2}
\end{center}
\end{figure}

    \textit{Subcase 2.1.2.1:} there exists $w\in V_1$ with $I_{\overline z}\sm N(w)\ne\emptyset$. Let $u_1\in I_{\overline z}\sm N(w)$. Since $|I_z|\ge 2$, we can find $u_2,u_3\in I_z$ with $u_2\sim w$ and $u_3\not\sim w$. The $\overline{K_3}$ $\{u_1,u_2,u_3\}$ is good with respect to $(w,z,x)$.

    \textit{Subcase 2.1.2.2:} $e(V_1,I_{\overline z})=|V_1||I_{\overline z}|$. Now we claim that $V_1=\emptyset$. Otherwise, there exist $w\in V_1$ and $u\in I_z$ such that $w\not\sim u$. Then $N(\overline w,\overline u)\subseteq V_\emptyset$, so $\overline\lambda\le|V_{\emptyset}|$. On the other hand, choosing $u_1,u_2\in I_z$, we have $|N(\overline{u_1},\overline{u_2})|\ge|V_{\emptyset}|+|I_{\overline z}|\ge\overline\lambda+1$, a contradiction. This proves the claim. Now for any $u\in I_z$, we have $k=|N(u)|=|V_0|$. But for any $u\in I_{\overline z}$, we have $N(u)\subseteq V_0$ and $z\in V_0\sm N(u)$, so $k=|N(u)|\le|V_0|-1$, a contradiction.

    \textit{Subcase 2.2:} for all $z\in V$, we have $d_I(z)\in \{0,1,s\}$. For $i\in\{0,1,s\}$, set
    $$V_i=\{z\in V\sm I:d_I(v)=i\}.$$
    Note that $V_0,V_s\ne\emptyset$. If $u_1,u_2\in I$, then $N(u_1,u_2)=V_s$, so $|V_s|=\mu$. For $z\in V_1$, define $f(z)$ to be the unique $u\in I$ such that $z\sim u$. For each $u\in I$, we have $|f^{-1}(\{u\})|=k-\mu>0$; in particular, $|V_1|\ge s$.

    \textit{Subcase 2.2.1:} there exists $z\in V_1$ such that $0<d_{V_s}(z)<\mu$. Then we can choose $w_1,w_2\in V_s$ such that $z\sim w_1$ and $z\not\sim w_2$. Also, choose $u_1,u_2\in I\sm\{f(z)\}$, possible since $s\ge 3$. The $\overline{K_3}$ $\{z,u_1,u_2\}$ is good with respect to $(f(z),w_2,w_1)$.

    \textit{Subcase 2.2.2:} for all $z\in V_1$, we have $d_{V_s}(z)\in\{0,\mu\}$. For $i\in\{0,\mu\}$, set
    $$V_{1,i}=\{z\in V_1:d_{V_s}(z)=i\}$$
    (see Figure \ref{Figure Subcase 2.2.2}). Now we claim that, for all $u\in I$, $f^{-1}(\{u\})\cap V_{1,0}\ne\emptyset$. Suppose to the contrary that $f^{-1}(\{u\})\subseteq V_{1,\mu}$. Then for any $z_1,z_2\in f^{-1}(\{u\})$, we have $N(z_1,z_2)\supseteq V_0\cup\{u\}$, so $d(z_1,z_2)\ge\mu+1$ and thus $z_1\sim z_2$. Thus, $f^{-1}(\{u\})$ is a clique, and we have $\lambda=d(z_1,u)=\mu+(k-\mu-1)=k-1$, contradicting primitivity.

\begin{figure}
        \begin{center}
\begin{tikzpicture}[scale=1.5]
    \draw[very thick] (1,0)--(2.5,1);
    \draw[dashed, very thick] (5.5,1)--(7,0);
    \draw[dashed, very thick] (1,0.2)--(5,-0.4);
    \draw[very thick] (1,0)--(3,-0.7);
    \draw[fill=white] (4,1) ellipse (2 and .5) node {$s$};
    \node at (4,1.7) {$I$}; 
    \draw[fill=white] (1,0) ellipse (.5 and .5) node {$\mu$};
    \node at (1,-0.7) {$V_s$};
    \draw[fill=white] (3,-0.7) ellipse (.9 and .5);
    \node at (3,-1.4) {$V_{1,\mu}$};
    \draw[fill=white] (5,-0.7) ellipse (.9 and .5);
    \node at (5,-1.4) {$V_{1,0}$};
    \draw[fill=white] (7,0) ellipse (.5 and .5);
    \node at (7,-0.7) {$V_0$};
\end{tikzpicture}
\end{center}
\caption{The structure of $G$ in Subcase 2.2.2.} \label{Figure Subcase 2.2.2}
\end{figure}
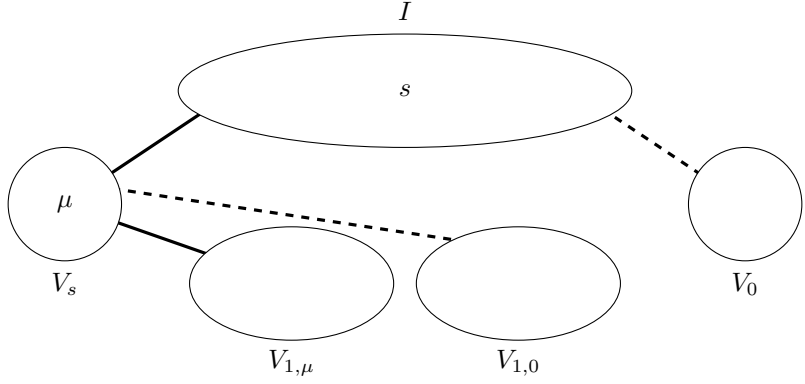

    \textit{Subcase 2.2.2.1:} $V_{1,\mu}\ne\emptyset$. Let $z_1\in V_{1,\mu}$ and set $u_1=f(z_1)$. Choose $z_2\in f^{-1}(\{u_1\})\cap V_{1,0}$, $u_2\in I\sm\{u_1\}$, and $z_3\in f^{-1}\{u_2\}\cap V_{1,0}$.

    \textit{Subcase 2.2.2.1.1:} there exist $w_1,w_2\in V_s$ with $w_1\sim w_2$. Then the $K_3$ $\{w_1,w_2,u_1\}$ is good with respect to $(z_3,z_2,z_1)$.

    \textit{Subcase 2.2.2.1.2:} $V_s$ is an independent set. 
    Choose $z\in V_{1,0}$. For each $u\in I$, we have $N(z,u)\subseteq f^{-1}(\{u\})$, so $k=d(z)\ge\sum_{u\in I\sm\{f(z)\}}d(z,u)=(s-1)\mu$, so $k\ge 2\mu$ as $s\ge 3$.

    \textit{Subcase 2.2.2.1.2.1:} $\mu\ge 2$. Then there exist $w_1,w_2\in V_s$ with $w_1\not\sim w_2$. Choose $z_1\in V_{1,0}$ and consider the $\overline{K_3}$ $\{w_1,w_2,z_1\}$. Note that $f(z_1)\in N(w_1,w_2,z_1)$. Thus, $d(z_1,\overline{w_1},\overline{w_2})=k-2\mu+d(z_1,w_1,w_2)\ge k-2\mu+1\ge 1$, so there exists some $z_2\in N(z_1,\overline{w_1},\overline{w_2})$. Choose $u\in I\sm\{f(z_1)\}$. Then $\{z_1,w_1,w_2\}$ is good with respect to $(z_2,u,f(z_1))$.

    \textit{Subcase 2.2.2.1.2.2:} $\mu=1$. Since $(\lambda,\mu)\ne(0,1)$, this implies that $\lambda\ge 2$, by Theorem \ref{Theorem Hoffman-Singleton}, Lemma \ref{Lemma vk11}, and $v\ge 6401$. For any $z\in V$, $N(z)$ is a disjoint union of cliques of order $\lambda+1\ge 3$. Since $G$ is primitive, there must be at least 2 such cliques by Lemma \ref{Lemma SRG positive}, so $k\ge2(\lambda+1)>2\lambda$. Consider the $K_3$ $\{z_1,z_2,z_3\}$ contained in one of these cliques, and let $z_4$ be in a different one of these cliques. Now $d(z_1,\overline{z_2},\overline{ z_3})\ge k-|N(z_1,z_2)|-|N(z_1,z_3)|=k-2\lambda>0$, so there exists $z_5\in N(z_1,\overline{z_2},\overline{z_3})$. Then $\{z_1,z_2,z_3\}$ is good with respect to $(z_4,z_5,z)$.

    \textit{Subcase 2.2.2.2:} $V_{1,\mu}=\emptyset$. Then $e(V_s,V_1)=0$. For any $z\in V_1$ and $u\in I\sm\{f(z)\}$, we have $N(z,u)\subseteq V_1$, so $|N(z)\cap f^{-1}(\{u\})|=\mu$. Summing over all $u\in I\sm\{f(z)\}$, we obtain $k=|N(z)|\ge\mu(s-1)$. Now choose $w\in V_s$ and $u\in I$. Since $N(w,u)\subseteq V_s$, we have $\lambda=d(w,u)\le|V_s|-1=\mu-1$. So we have $k\ge\mu(s-1)\ge(\lambda+1)(s-1)$. Since $s\ge 3$, we have $k\ge2\mu$ and $k>2\lambda$. Now if $\mu=1$, then $\lambda\le\mu-1=0$, contradicting $(\lambda,\mu)\ne(0,1)$. So $\mu\ge 2$.

    \textit{Subcase 2.2.2.2.1:} $\mu=2$. Let $z\in V_1,u\in I\sm\{f(z)\}.$ Since $\mu=2$ and and $z\in V_{1,0}$, there exist $z_1,z_2\in N(z,u)$. Now $d(z_1,z_2)\ge 2>\lambda$ so $z_1\not\sim z_2$. Choose $w\in V_s$, so that $\{w,z_1,z_2\}$ is a $\overline{K_3}$ which is good with respect to $(f(z),z,u)$.

    \textit{Subcase 2.2.2.2.2:} $\mu\ge 3$. 

    \textit{Subcase 2.2.2.2.2.1:} there exist $w_1,w_2\in V_s$ with $w_1\not\sim w_2$. Choose $z\in V_1$. Since $f(z)\in N(w_1,w_2,z)$, we have $d(z,\overline{w_1},\overline{w_2})\ge k-2\mu+d(z,w_1,w_2)\ge k-2\mu+1>0$, so there exists $z_1\in N(z,\overline{w_1},\overline{w_2}$. Let $u_1\in I\sm\{f(z)\}$. Then the $\overline{K_3}$ $\{w_1,w_2,z\}$ is good with respect to $(z_1,u_1,f(z))$.

    \textit{Subcase 2.2.2.2.2.2:} $V_s$ is a clique. Since $\mu\ge 3$, we can take $w_1,w_2,w_3\in V_s$. Let $u_1,u_2\in I$ and let $z_1\in f^{-1}(\{u_1\})$, $z_2\in f^{-1}(\{u_2\})$. Then the $K_3$ $\{w_1,w_2,u_1\}$ is good with respect to $(z_2,z_1,w_3)$.
\end{proof}

We conclude this section with some general upper bounds, which have proven useful in analyzing specific SRGs. Part (1) of Proposition \ref{Proposition misc bounds}, and Part (2) without the equality case, hold for any $k$-regular graph.

\begin{prop} \label{Proposition misc bounds}
    Suppose that $G$ is an $\mathrm{SRG}(v,k,\lambda,\mu)$ with VC-dimension $n\ge 2$.
    \begin{enumerate}
        \item $2^n\le v$. If equality holds, then $k=v/2$.
        \item $2^{n-1}\le k$. If equality holds, then $\lambda=k/2$ or $\mu=k/2$.
        \item $2^{n-2}\le\max(\lambda,\mu)$.
        \item $2^{n-1}\le v-k$. If equality holds, then $3k=v+2\lambda$ or $3k=v+2\mu$.
        \item $2^{n-2}\le k-\min(\lambda,\mu)$.
        \item $2^{n-2}\le v-2k+\max(\lambda,\mu)$.
    \end{enumerate}
\end{prop}
\begin{proof}
    Assume that $S\subseteq V$ is a shattered set of size $n$.

    (1) For every $T\subseteq S$, there exists $x_T\in V$ such that $N(x_T)\cap S=T$. If $T\ne T'$ then $x_{T}\ne x_{T'}$. Thus,
    $$v\ge |\{x_T:T\subseteq S\}|=|\{T:T\subseteq S\}|=2^n.$$
    If equality holds, then $V=\{x_T:T\subseteq S\}$. For all $y\in S$,
    $$k=d(y)=|\{x\in V:x\sim y\}|=|\{T\subseteq S:x_T\sim y\}|=|\{T\subseteq S:y\in T\}|=2^{n-1}=v/2.$$

    (2) Fix $y\in S$. For every $T\subseteq S\sm\{y\}$, there exists $x_T\in V$ such that $N(x_T)\cap S=T\cup\{y\}$. Then
    $$k\ge|\{x_T:T\subseteq S\sm\{y\}\}|=2^{n-1}.$$
    If equality holds, then $N(y)=\{x_T:T\subseteq S\sm\{y\}\}$. Fix $z\in S\sm\{y\}$. Then
    $$d(y,z)=|\{x\in N(y):x\sim z\}|=|\{T\subseteq S\sm\{y\}:z\in T\}|=2^{n-2}=k/2.$$
    But $d(y,z)\in\{\lambda,\mu\}$.

    (3) Fix $y,z\in S$. For every $T\subseteq S\sm\{y,z\}$, there exists $x_T\in V$ such that $N(x_T)\cap S=T\cup\{y,z\}$. Then
    $$\max(\lambda,\mu)\ge d(y,z)
    \ge|\{x_T:T\subseteq S\sm\{y,z\}\}|=2^{n-2}.$$

    (4) Fix $y\in S$. For every $T\subseteq S\sm\{y\}$, there exists $x_T\in V$ such that $N(x_T)\cap S=T$. Then
    $$v-k=d'(\overline y)\ge|\{x_T:T\subseteq S\sm\{y\}\}|=2^{n-1}.$$
    If equality holds, then $N'(\overline y)=\{x_T:T\subseteq S\sm\{y\}\}$. Fix $z\in S\sm\{y\}$. Then
    $$d'(\overline y,z)=|\{x\in N'(\overline y):x\sim z
    \}=\{x_T:z\in T\}=2^{n-2}=(v-k)/2.$$
    But $d'(\overline y,z)\in\{k-\lambda,k-\mu\}.$

    (5) Fix $y,z\in S$. For every $T\subseteq S\sm\{y,z\}$, there exists $x_T\in V$ such that $N(x_T)\cap S=T\cup\{y\}$. Then
    $$k-\min(\lambda,\mu)\ge d'(y,\overline z)\ge|\{x_T:T\subseteq S\sm\{y,z\}\}=2^{n-2}.$$

    (6) Fix $y,z\in S$. For every $T\subseteq S\sm\{y,z\}$, there exists $x_T\in V$ such that $N(x_T)\cap S=T$. Then
    $$v-2k+\max(\lambda,\mu)\ge d'(\overline y,\overline z)\ge|\{x_T:T\subseteq S\sm\{y,z\}\}|=2^{n-2}.$$
\end{proof}

We briefly discuss the usefulness of the bounds in Proposition \ref{Proposition misc bounds}. We determined the VC-dimension of each SRG order at most 28 (see Section \ref{Section Small SRGs}). In exactly 15 of the 23 possible parameter sets, at least one of the bounds is tight. The specific bound giving the best value varies depending on the SRG, and sometimes the equality cases are needed. With respect to infinite families of graphs, Proposition \ref{Proposition misc bounds} will sometimes improve the best known upper bounds. For example, the equality case of (2) implies that the VC-dimension of the Paley graph of order $q$ is strictly smaller than $\log_2(q)$ (cf. \cite{mcdonald2025vc}).
    
\subsection{Latin squares} \label{Section Latin Squares}

A \textit{Latin square} of order $m$ is an $m\times m$ matrix with entries in $[m]:=\{1,2,\ldots,m\}$ such that each row and each column contains each entry exactly once. A Latin square $L$ gives rise to a graph $G(L)$ with vertex set $\{(i,j,L_{i,j}):i,j\in[m]\}$, where $(i,j,L_{ij})\sim(k,\ell,L_{k\ell})$ if and only if $i=k$, $j=k$, or $L_{ij}=L_{k\ell}$. The graph $G(L)$ is strongly regular with parameters $(m^2,3(m-1),m,6)$. In this section we will determine the VC-dimension of sufficiently large Latin square graphs. The following lemma, useful to simplify our bookkeeping, well-known. We include the proof for clarity.

\begin{lemma} \label{Lemma Latin square transformations}
Let $L$ be a Latin square and suppose that $L'$ is obtained from $L$ by one of the following operations:
\begin{enumerate}
    \item permuting the rows;
    \item permuting the columns;
    \item permuting the entries;
    \item taking the transpose;
    \item interchanging rows and entries (i.e., putting the entry $i$ in position $(L_{i,j},j)$);
    \item interchanging columns and entries (i.e., putting the entry $j$ in position $(i,L_{i,j})$).
\end{enumerate}
Then $G(L)\simeq G(L')$.
\end{lemma}
\begin{proof}
Let $\pi$ be a permutation of $[m]$. The Latin square obtained by applying $\pi$ to the rows/columns/entries contains $L_{\pi(i),j}$/$L_{i,\pi(j)}$/$\pi(L_{i,j})$ in the $(i,j)$-entry. In the case of the rows, let $\phi:V(G(L))\to V(G(L'))$ be given by $\phi((i,j,L_{i,j}))=(\pi^{-1}(i),j,L_{i,j})$. Then $(i,j,L_{i,j})\sim_{G(L)}(i',j',L_{i',j'})$ if and only if $i=i'$, $j=j'$, or $L_{i,j}=L_{i',j'}$; and $(\pi^{-1}(i),j,L_{i,j})\sim_{G(L')}(\pi^{-1}(i'),j',L_{i',j'})$ if and only if $\pi^{-1}(i)=\pi^{-1}(i')$, $j=j'$, or $L_{i,j}=L_{i',j'}$. But $\pi^{-1}(i)=\pi^{-1}(i')$ if and only if $i=i'$, so the two adjacencies are equivalent, and $\phi$ is an isomorphism. The column case is handled similarly. In the case of the entries, let $\phi:V(G(L))\to V(G(L'))$ be given by $\phi((i,j,L_{i,j}))=(i,j,\pi(L_{i,j}))$. It is easy to check that $\phi$ is an isomorphism.

For (4), we use the isomorphism $(i,j,L_{i,j})\mapsto(j,i,L_{i,j})$. For (5), we use the isomorphism $(i,j,L_{i,j})\mapsto (L_{i,j},j,i)$. This map is bijective because the $j^{th}$ column of $L$ contains each entry exactly once, and it is immediate that it is an isomorphism. Part (6) follows similarly.
\end{proof}

We use Lemma \ref{Lemma Latin square transformations} freely in the rest of this section. Note that, when we speak of the order of a Latin square graph, we mean the \textit{square} of the order of the Latin square. We start by examining the smallest nontrivial case, when the graph has order 16. There are only two nonisomorphic Latin square graphs of order 16 (see, e.g., the parameter table in \cite{brouwer2022strongly}).

\begin{eg} \label{Example LS 1}
    Consider the Latin square graph arising from $\mathbb Z_4$:
    $$\begin{bmatrix}
        1&2&3&4\\
        2&3&4&1\\
        3&4&1&2\\
        4&1&2&3
    \end{bmatrix}.$$
    One can check that $(1,1,1),(1,2,2)$, and $(3,2,4)$ form a shattered set of size 3. By Proposition \ref{Proposition misc bounds}(4), the VC-dimension is exactly 3.
\end{eg}

\begin{eg} \label{Example LS 2}
    The other Latin square of size 4 arises from $\mathbb Z_2\times\mathbb Z_2$:
    $$\begin{bmatrix}
        1&2&3&4\\
        2&1&4&3\\
        3&4&1&2\\
        4&3&2&1
    \end{bmatrix}.$$
    Again one can check that positions $(1,1,1)$, $(1,2,2)$, and $(3,2,4)$ form a shattered set of size 3. By Proposition \ref{Proposition misc bounds}(4), the VC-dimension is exactly 3.
\end{eg}

\begin{prop}
    The VC-dimension of any Latin square graph of order at least $16$ is at least 3.
\end{prop}
\begin{proof}
    Let $L$ be a Latin square of order $m\ge 4$ and let $G=G(L)$. If $m=4$ we may consult Examples \ref{Example LS 1} and \ref{Example LS 2}, so assume $m\ge 5$. Let row 1 of $L$ be $(1,2,\ldots,m)$. For $j\in\{1,2,3\}$, set $x_j=(1,j,L_{1j})=(1,j,j)$. We claim that $\{x_1,x_2,x_3\}$ is shattered. Since $\{x_1,x_2,x_3\}$ is a $K_3$, we may apply Lemma \ref{Lemma triangle}. Note that $L_{i4}=5$ for some $i\ne 1$, so set $y_1=(i,4,5)$. Also, $L_{i1}=4$ for some $i\ne 1$, so set $y_2=(i,1,4)$. Finally, set $y_3=(1,4,4)$. Then $\{x_1,x_2,x_3\}$ is good with respect to $(y_1,y_2,y_3)$.
\end{proof}

\begin{prop}
    The VC-dimension of any Latin square graph is at most 4.
\end{prop}
\begin{proof}
    Suppose to the contrary that $X$ is a shattered set of size 5. We may assume $X\subseteq N((1,1,1))$. Let $R=\{(i,j,\ell)\in X:i=1\}$, $C=\{i,j,\ell\in X:j=1\}$, $E=\{(i,j,\ell\in X:\ell=1\}$ so that $X=R\sqcup C\sqcup E$. Set $r=|R|,c=|C|,e=|E|$. By Lemma \ref{Lemma Latin square transformations}, we may assume $r\ge c\ge e$.

    \textit{Case 1:} $(r,c,e)=(2,2,1)$. We may assume that the shattered set consists of the bold entries in the following principal submatrix
    $$\begin{bmatrix}
        1&\textbf{2}&\textbf{3}\\
        \mathbf{a}&*&*\\
        \mathbf{b}&*&*
    \end{bmatrix}$$
    along with an entry equal to 1, somewhere in the Latin square. Label these vertices $x_2,x_3,x_a,x_b,x_1$. Since $N'(x_2,x_3,x_a,x_b,\overline{x_1})\ne\emptyset$, we must have $\{a,b\}\cap\{2,3\}\ne\emptyset$ and $L_{ij}\in\{a,b\}\cap\{2,3\}$ for some $(i,j)\in\{2,3\}\times\{2,3\}$. Without loss of generality, we have
    $$\begin{bmatrix}
        1&\textbf{2}&\textbf{3}\\
        \mathbf{2}&*&*\\
        \mathbf{b}&*&2
    \end{bmatrix}.$$
    Since $N'(x_2,x_3,x_a,x_1,\overline{x_b})\ne\emptyset$, this forces WLOG
        $$\begin{bmatrix}
        1&\textbf{2}&\textbf{3}&*\\
        \mathbf{2}&3&*&\mathbf{1}\\
        \mathbf{b}&*&2&*
    \end{bmatrix}$$
    and moreover $b\ne 3$, so we may assume $b=4$. Let $(i,j,\ell)\in N'(x_2,x_3,x_1,x_b,\overline{x_a})$. If $i\ne 1$ then we must have $i=2$ and $j=3$, but then $\ell\ne 2$, contradicting $(i,j,\ell)\sim x_2$. Thus, $i=1$, forcing
      $$\begin{bmatrix}
        1&\textbf{2}&\textbf{3}&4\\
        \mathbf{2}&3&*&\mathbf{1}\\
        \mathbf{4}&*&2&*
    \end{bmatrix}.$$
    Now let $(i',j',\ell')\in N'(x_2,x_a,x_b,x_1,\overline{x_3})$. If $\ell'=2$ then we must have $(i',j')=(3,4)$, which is impossible since $L_{3,3}=2$. So $\ell'\ne 2$, so $(i',j',\ell')\sim x_2,x_a$ forces $(i',j')\in\{(1,1),(2,2)\}$, and either case leads to a contradiction.

    \textit{Case 2:} $(r,c,e)=(3,1,1)$. The shattered set takes the form
    $$\begin{bmatrix}
        1&\textbf{2}&\textbf{3}&\textbf{4}\\
        \textbf{a}&*&*&*
    \end{bmatrix}$$
    along with a 1 appearing somewhere in the Latin square. Label the vertices in the shattered set $x_2,x_3,x_4,x_a,x_1$. Note that $N'(x_2,x_3,x_4,x_a,\overline{x_1})\ne\emptyset$ forces 
    $$\begin{bmatrix}
        1&\textbf{2}&\textbf{3}&\textbf{4}&5\\
        \textbf{5}&*&*&*&*
    \end{bmatrix}.$$
    with $x_1$ not in column 5. Since $N'(x_2,x_3,x_4,x_1,\overline{x_a})\ne\emptyset$, we may assume that $L_{1,6}=6$ and $x_1$ belongs to column 6. Then $N'(x_2,x_3,x_a,x_1,\overline{x_4})\ne\emptyset$ forces
    $$\begin{bmatrix}
        1&\textbf{2}&\textbf{3}&\textbf{4}&5&6\\
        \textbf{5}&3&*&*&*&\textbf{1}
    \end{bmatrix}$$
    up to permuting columns $2,3$ and entries $2,3$. Now the facts $N'(x_2,x_4,x_a,x_1,\overline{x_3})\ne\emptyset$ and $N'(x_3,x_4,x_a,x_1,\overline{x_2})\ne\emptyset$ force
    $$\begin{bmatrix}
        1&\textbf{2}&\textbf{3}&\textbf{4}&5&6\\
        \textbf{5}&3&4&2&*&\textbf{1}
    \end{bmatrix}.$$
    But then we see $N(x_2,x_3,x_a,\overline{x_4},\overline{x_1})=\emptyset$, a contradiction.

    \textit{Case 3:} $(r,c,e)=(3,2,0)$. The shattered set takes the form
    $$\begin{bmatrix}
        1&\textbf{2}&\textbf{3}&\textbf{4}\\
        \textbf{a}&*&*&*\\
        \textbf{b}&*&*&*
    \end{bmatrix}.$$
    Since $N'(x_2,x_3,x_4,x_a,\overline{x_b})\ne\emptyset$ and $N'(x_2,x_3,x_4,x_b,\overline{x_a})\ne\emptyset$, we may assume that $a=5$ and $b=6$. Now let $(i,j,\ell)\in N'(x_2,x_3,x_a,x_b,\overline{x_4})$. We have $i\ne 1$, so $\ell\in\{2,3\}$. Also, $j\ne 1$, so $\ell\in\{5,6\}$, a contradiction.

    \textit{Case 4:} $r\ge 4$. Let $x_1,x_2,x_3,x_4$ be 4 elements of the shattered set all in row 1. Then $N(x_1,x_2,x_3,\overline{x_4})=\emptyset$, contradicting assumption.
\end{proof}

\begin{thm} \label{Theorem Latin square} The VC-dimension of any Latin square graph of order at least $225$ is 4.
\end{thm}
\begin{proof}
    It only remains to prove the lower bound. Let $L$ be a Latin square of order $m\ge\sqrt{225}=15$. Let row 1 be $(1,2,\ldots,m)$. The number of entries in columns 1, 2, 3 equal to 1, 2, or 3 is $3+3\cdot 2$, and they appear in at most $1+3\cdot 2=7$ distinct rows. Since $m\ge 8$, we can assume that row 2 is $(4,5,6,*,\ldots,*)$. The number of entries in columns 1, 2, 3 equal to 1, 2, 3, 4, 5, or 6 is at most $6+6\cdot 2$, and they appear in at most $2+6\cdot 2=14$ distinct rows. Since $m\ge 15$, we can assume that row 3 is $(7,8,9,*,\ldots,*)$. Set $x_2=(1,2,2)$, $x_3=(1,3,3)$, $x_4=(2,1,4)$, $x_7=(3,1,7)$. We claim that $X=\{x_2,x_3,x_4,x_7\}$ is shattered. 
    $$\begin{bmatrix}
    1&\mathbf{2}&\mathbf{3}\\
    \mathbf{4}&5&6\\
    \mathbf{7}&8&9
\end{bmatrix}$$
    Row 4 contains $m-9\ge 6$ entries outside $[9]$, so set $y_0=(4,j,\ell)$ for some $\ell\not\in[9]$ and $j\not\in[3]$. Then $N_X(y_0)=\emptyset$. Row 2 contains $m-9\ge 6$ entries outside $[9]$, so set $y_2=(2,j,\ell)$ for some $\ell\not\in[9]$ and $j\not\in[3]$. Then $N_X(y_2)=\{x_2\}$. Similarly, we can find $y_3,y_4,y_7$ such that $N_x(y_\ell)=\{x_\ell\}$, for $\ell\in\{3,4,7\}$. We have $N_X((2,2,5))=\{x_2,x_4\}$, $N_X((2,3,6))=\{x_3,x_4\}$, $N_X((3,2,8))=\{x_2,x_7\}$, and $N_X((3,3,9))=\{x_3,x_7\}$. Row 1 contains $m-9\ge 6$ entries outside $[9]$, so set $y_{2,3}=(1,j,\ell)$ for some $\ell\not\in[9]$. Then $N_X(y_{2,3})=\{x_2,x_3\}$. Similarly, there exists $y_{4,7}$ such that $N_X(y_{4,7})=\{x_4,x_7\}$. There exists $j\ge 4$ with $L_{1j}=4$, so setting $z_7=(1,j,4)$, we have $N_X(z_7)=X\sm\{x_7\}$. Similarly, we can find $z_i$ such that $N_X(z_i)=X\sm\{x_i\}$, for $i\in\{2,3,4\}$. Finally, we have $N_X((1,1,1))=X$.
\end{proof}

\subsection{Orthogonal arrays}

An \textit{orthogonal array} $\mathrm{OA}(k,m)$ is a $k\times m^2$ matrix with entries in $[m]$ such that, in any $2\times m^2$ submatrix, all possible ordered pairs in $[m]\times[m]$ appear as columns. The graph of an $\mathrm{OA}(k,m)$ is an $\mathrm{SRG}(m^2,k(m-1),m-2+(k-1)(k-2),k(k-1))$ (see e.g. \cite{godsil2013algebraic}).

\begin{thm} \label{Theorem OA upper bound}
    Let $G$ be the graph of an $\mathrm{OA}(k,m)$. The VC-dimension of $G$ is at most $k(k-1)$. 
\end{thm}
\begin{proof}
    Let $G$ be the graph of an $\mathrm{OA}(k,m)$, and suppose that $S\subseteq V$ is a shattered set of size $n$.

    \textit{Case 1:} $S$ is not a clique. Then there exist $x,y\in S$ with $x\not\sim y$. For each $T\subseteq S\sm\{x,y\}$, there exists some $z_T\in V$ such that $N_S(z_T)=T\cup\{x,y\}$. The $z_T$ must all be distinct. Thus,
    we have
    $$2^{n-2}=|\{z_T:T\subseteq S\sm\{x,y\}\}|\le|N(x,y)|=\mu=k(k-1)$$
    and so $n\le\log_2(k(k-1))+2$. Finally, since $k\ge 3$, we have $k(k-1)\ge 4$ and so $\log_2(k(k-1))+2\le k(k-1).$

    \textit{Case 2:} $S$ is a clique. Then for any two columns $c,d$ corresponding to vertices in $S$, we may color the edge $cd$ with the unique index $i$ such that $c_i=d_i$. This yields an edge-coloring of the complete graph on $S$, using colors in $[k]$. By a result of Gy\'arf\'as \cite{gyarfas1977partition}, there is a monochromatic component $C$ of order at least $n/(k-1)$; suppose that this component is monochromatic in color $i^*$. By connectivity of $C$, we have $c_{i^*}=d_{i^*}$ for any $c,d\in C$; let $j^*$ be this common symbol in position $i^*$, such that $c_{i^*}=j^*$ for all $c\in C$. Let $T\subseteq C$ be an arbitrary subset of size $|C|-1$. Since $S$ is shattered, there exists $x\in V$ such that $N_C(x)=T$. Since there exists $c\in C$ with $x\not\sim c$, we have $x_{i^*}\ne j^*$. Thus, for all $d\in T$, we have $x_{i}=d_i$ for some $i\ne i^*$, and moreover this choice of $i$ must be distinct for all $d\in T$. Since there are $k-1$ such $i$, it follows that $n/(k-1)-1\le |C|-1=|T|\le k-1$ and so $n\le k(k-1)$.
\end{proof}

There exist $\mathrm{OA}(m/2,m)$ whose graphs have unbounded VC-dimension as $m\to\infty$ \cite{baker2008graphs}. Furthermore, the upper bound of Theorem \ref{Theorem OA upper bound} is not tight for $k=3$, in view of Theorem \ref{Theorem Latin square}. It seems difficult to provide tight lower and upper bounds for arbitrary values of $k$. We were not even able to generalize the lower bound in Theorem \ref{Theorem Latin square} to general orthogonal arrays.

\subsection{Steiner graphs}

A $2$-$(v,k,1)$ \textit{design} is a pair $(\mathcal P,\mathcal B)$ of sets (\textit{points} and \textit{blocks}), where $\mathcal B\subseteq 2^\mathcal P$, with the properties:
\begin{enumerate}
    \item $|\mathcal P|=v$,
    \item $|B|=k$ for all $B\in\mathcal B$,
    \item for any distinct $p,q\in\mathcal P$, there is a unique $B\in\mathcal B$ such that $p,q\in B$.
\end{enumerate}

Given a $2$-$(v,k,1)$ design $(\mathcal P,\mathcal B)$, we define a graph $G=(V,E)$ where $V$ is the set of $b={v\choose 2}/{k\choose 2}$ blocks, and $B\sim C$ if and only if $B\cap C\ne\emptyset$. The graph $G$ (called a \textit{Steiner graph}) is strongly regular (see e.g. \cite{brouwer2022strongly}, Section 8.5.) The expressions for the parameters are somewhat complicated, but we only need the fact that $\mu=k^2$ (here $k$ denotes the block size, not valency). We show that the Steiner graphs with fixed block size have bounded VC-dimension.

\begin{thm} \label{Theorem Steiner} Let $G$ be the Steiner graph of a $2$-$(v,k,1)$ design. Then the VC-dimension of $G$ is at most $k(k+1)$.
\end{thm}
\begin{proof}
    Suppose that $S\subseteq V$ is a shattered set of size $n$. 

    \textit{Case 1:} $S$ is not a clique. Then there exist $x,y\in S$ with $x\not\sim y$. For each $T\subseteq S\sm\{x,y\}$, there exists some $z_T\in V$ such that $N_S(z_T)=T\cup\{x,y\}$. Such of the $2^{n-2}$ possible $T$ corresponds to a distinct $z_T$, and each $z_T$ is a common neighbor of $x$ and $y$. Thus, $2^{n-2}\le\mu=k^2$, so $n\le 2\log_2(k)+2$. Since $k\ge 3$, we have $k^2\ge 4$, so $4k^2\le 2^{k^2}\le 2^{k(k+1)}$ and so $2\log_2(k)+2=\log_2(4k^2)\le k(k+1)$.

    \textit{Case 2:} $S$ is a clique. Fix some $x\in S$. Then there is a point $p\in x$ such that at least $n/k$ vertices in $S$ contain $p$; thus, let $C$ be a set of at least $n/k$ vertices, each of which contains $p$. Observe that any two vertices in $C$ intersect only in $p$. Let $T\subseteq C$ be a subsets of size $|C|-1$. Since $S$ is shattered, there exists some $y\in V$ such that $N_C(y)=T$. Since $N_C(y)\ne C$, we have $p\not\in y$. Since the vertices in $C$ intersect only in $p$, $y$ can intersect at most $k$ of them, so we have $k\ge |T|=|C|-1\ge n/k-1$, so $n\le k(k+1)$.
\end{proof}

Our next result generalizes Theorem 4.2.3 of \cite{benediktsson2021model} as applied to triangular graphs.

\begin{thm} \label{Theorem Steiner lower bound}
    Let $G$ be the Steiner graph of a $2$-$(v,k,1)$ design. If $v>\max(3k(k-1)^3,4k^2)$ then the VC-dimension of $G$ is at least 4.
\end{thm}
\begin{proof}
    Fix a point $p$ and two distinct blocks $B_1,B_2$ each containing $p$. The number of blocks containing $p$ is $(v-1)/(k-1)$, so this is possible if $v\ge 2k-1$. Let $\mathcal I_{1,2}:=\{B\in\mathcal B:B\cap(B_1\cup B_2)\ne\emptyset\text{ and }p\not\in B\}$. We have $|\mathcal I_{1,2}|\le (k-1)^2$, and so
    $$\left|\left\{B\in\mathcal B:B\cap\left(\bigcup_{C\in\mathcal I_{1,2}}C\right)\ne\emptyset\text{ and }p\in B\right\}\right|\le(k-1)^2\cdot k\cdot 1.$$
    Thus, if $(v-1)/(k-1)>k(k-1)^2$ then there exists a block $B_3$ such that the implication
    $$(p\not\in B\wedge B\cap(B_1\cup B_2)\ne\emptyset)\Longrightarrow B\cap B_3=\emptyset$$
    holds for all $B\in\mathcal B$. By similar reasoning, if $(v-1)/(k-1)\ge 3k(k-1)^2$, then there exists a block $B_4$ such that, if $p\not\in B$ and $B$ intersects at least two of $B_1,B_2,B_3$, then $B\cap B_4=\emptyset$, for all $B\in\mathcal B$. Now $B_1,B_2,B_3,B_4$ satisfy that

    \begin{equation} \label{Equation Steiner construction}
    \text{if }B\in\mathcal B\text{ and }p\not\in B,\text{ then }B\text{ intersects exactly }0,1,\text{ or }2\text{ of the sets } B_1,B_2,B_3,B_4.
    \end{equation}

    Now we show that $S=\{B_1,B_2,B_3,B_4\}$ is a shattered set. The number of sets intersecting at least one block in $S$ is at most $4\cdot k\cdot (v-1)/(k-1)$. If $v>4k^2$, then we obtain $b>4k(v-1)/(k-1)$, so $N(\overline{B_1},\overline{B_2},\overline{B_3},\overline{B_4})\ne\emptyset$. Now the number of blocks intersecting both $B_1$ and one of $B_2,B_3,B_4$ which do not contain $p$ is $(k-1)\cdot 3(k-1)=3(k-1)^2$. If $v>3(k-1)^3$, then $(v-1)/(k-1)>3(k-1)^2$, and we obtain $N(B_1,\overline{B_2},\overline{B_3},\overline{B_4})$. Similarly we find blocks which intersect only $B_2$, only $B_3$, and only $B_4$. For distinct $B_i$ and $B_j$, choosing $q_i\in B_i\sm\{p\}$ and $q_j\in B_j\sm\{p\}$, we find a block $B$ which contain $q_i$ and $q_j$. Note that $p\not\in B$. By (\ref{Equation Steiner construction}), $B\in N(B_i,B_j,\overline{B_\ell},\overline{B_m})$, where $\{i,j,\ell,m\}=\{1,2,3,4\}$. Note that $B_1\in N(\overline{B_1},B_2,B_3,B_4)$ and similarly for the other 3-subsets of $S$. Finally, if $v>4k-3$, then $(v-1)/(k-1)>4$, so there exists $B\in\mathcal B$ with $p\in B$ and $B\not\in S$; then $B\in N(B_1,B_2,B_3,B_4)$.
\end{proof}

\subsection{Small SRGs} \label{Section Small SRGs}
By a computer search, we determined the VC-dimension of every SRG of order at most 28. The results are recorded in the table below. The column ``Number" refers to the number of non-isomorphic SRGs with the given parameters. We have not found two SRGs with the same parameters and different VC-dimensions, although there is no reason to believe that such a pair does not exist.

\[ \begin{array}{c|c|c}
\text{Parameters} & \text{Number} & \text{VC-dim} \\ \hline{}
    (5,2,0,1) & 1 & 2  \\
    (9,4,1,2) & 1 & 2  \\
    (10,3,0,1) & 1 & 2  \\
    (10,6,3,4) & 1 & 2  \\
    (13, 6, 2, 3) & 1 & 3  \\
    (15,6,1,3) & 1 & 3  \\
    (15,8,4,4) & 1 & 3 \\
    (16,5,0,2) & 1 & 3  \\
    (16,10,6,6) & 1 & 3  \\
    (16,6,2,2) & 2 & 3 \text{ (all)}  \\
    (16,9,4,6) & 2 & 3 \text{ (all)}  \\
    (17,8,3,4) & 1 & 3  \\
    (21,10,3,6) & 1 & 3  \\
    (21,10,5,4) & 1 & 4  \\
    (25,8,3,2) & 1 & 3  \\
    (25,16,9,12) & 1 & 3  \\
    (25,12,5,6) & 15 & 4 \text{ (all)} \\
    (26,10,3,4) & 10 & 3 \text{ (all)} \\
    (26,15,8,9) & 10 & 4 \text{ (all)} \\
    (27,10,1,5) & 1 & 4 \\
    (27,16,10,8) & 1 & 4 \\
    (28,12,6,4) & 4 & 4 \text{ (all)} \\
    (28,15,6,10) & 4 & 3 \text{ (all)} \\
\end{array} \]

\section{Conclusion}

Our results on $n$-e.c. graphs give a relatively complete understanding of the function $m(v,n)$. Indeed, for $n\ge 3$, the value of $m(v,n)$ is closely related to the minimum dimensions of an $n$-strong binary covering array, and the proofs of the best bounds for the two problems are basically analogous. For $n=2$, the value of $m(v,n)$ is determined up to an additive constant, although it may be interesting to investigate the structure of the extremal graphs. Future work may consider the optimal values of other graph parameters for $n-$e.c. graphs, such as minimum degree or eigenvalues. It would be interesting to see whether novel combinatorial structures such as the graph $D_\ell$ also arise in the context of these problems.

Our results on VC-dimension, by contrast, leave many natural open questions. The proof of Theorem \ref{Theorem VC=2} suggests that even very weak lower bounds will be difficult to obtain for general SRGs. We find it interesting that the family of graphs characterized by Theorem \ref{Theorem VC=2} is not known to be finite, yet contains only 3 known examples. The connection between VC-dimension and this fundamental open problem about SRGs may be essentially coincidental. Also intriguing is the usefulness of Ramsey-theoretic results in proving Theorems \ref{Theorem VC=2} and \ref{Theorem Steiner}, fitting into the broader interplay between graph Ramsey theory and VC-dimension. We did not prove any substantial results on conference graphs, and we did not attempt to apply eigenvalue methods to study the VC-dimension of SRGs, although they have proved extremely important in the theory of SRGs. These may be interesting directions for future work.

Several natural open problems bear mentioning. We no not expect the following problem to be difficult, but we do not know the minimum order of two such graphs.

\begin{prob}
    Find two strongly regular graphs with the same parameters and different VC-dimensions.
\end{prob}

Solutions to the following problems would probably be more difficult to obtain, and would substantially improve on our results. Firstly, we ask whether a version of Theorem \ref{Theorem VC=2} can be proved with no restriction on $v$.

\begin{prob}
    Characterize all strongly regular graphs with VC-dimension exactly 2.
\end{prob}

Secondly, we ask for the optimal constants in our results on orthogonal arrays and Steiner systems.

\begin{prob}
    For fixed $k$ and large enough $m$, determine the minimum and maximum possible VC-dimension of the graph of a $\mathrm{OA}(k,m)$.
\end{prob}

\begin{prob}
    For fixed $k$ and large enough $v$, determine the minimum and maximum possible VC-dimension of the Steiner graph of a 2-$(v,k,1)$ design.
\end{prob}

\section{Acknowledgments}
The authors would like to thank Michael Tait for helpful discussion of the paper \cite{pham2025vc} and Istv\'an Tomon for pointing out bounds relating a graph's VC-dimension and its smallest eigenvalue.

\section{Statement on AI usage}

We used GPT-5.5 to assist with literature review and search for known bounds on the function $m(v,n)$. GPT-5.5 pointed out the following facts which were useful in this paper: (1) Lemma \ref{Lemma minimum degree}, and (2) if $\mathcal F$ is a family of $t$-qualitatively independent sets and $t\ge 3$, then $A\cap B\ne A\cap C$ for any $A,B,C\in\mathcal F$. All other results presented in this paper were discovered, proved, and written up solely by the authors.

\printbibliography

\end{document}